\documentclass[11pt,reqno]{amsart}

\usepackage[text={400pt,660pt},centering]{geometry}
\usepackage{color,xcolor}
\usepackage{esint,amssymb}
\usepackage{graphicx}
\usepackage{mathtools}
\usepackage[colorlinks=true, pdfstartview=FitV, linkcolor=blue, citecolor=blue, urlcolor=blue,pagebackref=false]{hyperref}

\usepackage{bm}
\usepackage[shortlabels]{enumitem}
\usepackage[normalem]{ulem}
\usepackage{comment}
\usepackage{wasysym}

\usepackage{tikz-qtree}
\usepackage{bbm}
\usepackage{dsfont}
\usepackage{pgfplots}
\usetikzlibrary{decorations.pathmorphing}
\usepgfplotslibrary{fillbetween}
\usetikzlibrary{intersections,patterns}

\renewcommand{\emptyset}{\varnothing}
\newcommand{\E}{\mathbf{E}}
\renewcommand{\P}{\mathbf{P}}

\newcommand{\R}{\mathbb{R}}
\newcommand{\RR}{\mathbb{R}}
\newcommand{\1}{\mathbbm{1}}
\newcommand{\indc}{\mathds{1}}
\newcommand{\Z}{\mathbb{Z}}
\newcommand{\norm}[1]{\left|\left|#1\right|\right|}
\newcommand{\N}{\mathbb{N}}
\newcommand{\cons}{\sigma}

\DeclareMathOperator{\Bet}{Beta}

\numberwithin{equation}{section}

\newtheorem{thm}{Theorem}[section]
\newtheorem{lemma}[thm]{Lemma}
\newtheorem{prop}[thm]{Proposition}
\newtheorem{cor}[thm]{Corollary}

\theoremstyle{remark}
\newtheorem{remark}[thm]{Remark}
\theoremstyle{definition}
\newtheorem{define}[thm]{Definition}

\newcommand{\ve}{\varepsilon}

\newcommand{\vp}{\varphi}

\newcommand{\al}{\alpha}

\newcommand{\cT}{\mathcal{T}}
\newcommand{\cL}{\mathcal{L}}
\newcommand{\rC}{\mathrm{C}}
\newcommand{\rD}{\mathrm{D}}

\newcommand{\acm}{\ensuremath{\mathrm{ACM}(m,q,r,\mu)}}
\newcommand{\tmesh}{\Delta_t}
\newcommand{\xmesh}{\Delta_x}
\newcommand{\eps}{\varepsilon}
\newcommand{\constwo}{R}
\newcommand{\gfunc}{g}
\newcommand{\hfunc}{h}

\title[]{Asymmetric cooperative motion in one dimension}
\author[]{Louigi Addario-Berry}

\address{Department of Mathematics and Statistics, McGill University}
\email{louigi.addario@mcgill.ca}

\author[]{Erin Beckman}
\address{Department of Mathematics and Statistics, Concordia University}
\email{erin.beckman@concordia.ca}

\author[]{Jessica Lin}
\address{Department of Mathematics and Statistics, McGill University}
\email{jessica.lin@mcgill.ca}

\begin{document}


\begin{abstract}
We prove distributional convergence for a family of random processes on $\Z$ which we call asymmetric cooperative motions. The model generalizes the ``totally asymmetric q-lazy hipster random walk'' introduced in \cite{Addario-Berry2019}. We present a novel approach based on 
connecting a temporal recurrence relation satisfied by the cumulative distribution functions of the process to the theory of finite difference schemes for Hamilton-Jacobi equations \cite{Crandall1984}. We also point out some surprising lattice effects that can persist in the distributional limit, and propose several generalizations and directions for future research.
\end{abstract}

\maketitle

\section{Introduction}\label{s.intro}

\subsection{Description of the model and the main result}\label{ss.des}

The purpose of this paper is to prove distributional convergence for a family of random processes we term {\em cooperative motions}, and in so doing develop the connection, introduced in \cite{Addario-Berry2019}, between convergence of recursive distributional equations and numerical approximation of PDEs. The image underlying the name ``cooperative motion'' is this: a walker is attempting to perform a random walk on $\Z$ with initial distribution $\mu$ and step distribution $\nu$, but at each step needs the help of some fixed number $m\ge 1$ other individuals (independent walkers) in order to move. If $m$ additional, independent copies of the process all find themselves at the same location as the first walker, then the walker succeeds in taking a $\nu$-distributed step; otherwise it stays put.

One natural probabilistic formulation of this model is as a tree-indexed random process. Let $\cT$ be the complete rooted $(m+1)$-ary tree, with root labeled by $\emptyset$ and node $v$ having children $(vi,1 \le i \le m+1)$, so nodes at distance $d$ from the root are labeled by strings $c_1c_2\ldots c_d \in \{1,\ldots,m+1\}^d$. 
Write $\cT_n$ for the subtree of $\cT$ containing only nodes at distance $\le n$ from the root, and write $\cL_n$ for the leaves of $\cT_n$. 

Next, fix a probability distribution $\mu$ on $\Z \cup \{-\infty, \infty\}$ --- it turns out to be useful to allow $\pm \infty$ as initial positions --- 
and let $\rC^n=(C^n_v,v \in \cL_n)$ be iid, $\mu$-distributed random variables. Let $\rD=(D_v,v \in \cT_n\setminus \cL_n)$ be a second collection of iid $\Z$-valued random variables with some law $\nu$, independent of the  variables $C_v$. We may view $\cT_n$ as computing a function with inputs at the leaves, given by $\rC$, and output at the root, using as auxiliary randomness the elements of $\rD$, as follows. For $v \in \cT_n\setminus \cL_n$, recursively define
\[
C^n_v = 
\begin{cases}
C^n_{v1}+D_v & \mbox{ if }C_{v1}=C_{v2}=\ldots=C_{v(m+1)}\, ,\\	
C^n_{v1}		&\mbox{ otherwise.}
\end{cases}
\]
The output of the function is the random variable $C^n_{\emptyset}$ indexed by the root; this random variable has the distribution of a cooperative motion at time $n$. 

There is a second formulation, which is slightly less visual but is easier to work with, and is the one we use for the rest of the paper. 
Let $(D_n,n \ge 0)$ be a collection of independent, identically distributed integer random variables with common law $\nu$. Define a sequence $(X_n,n \ge 0)$ of extended real random variables as follows. Let $X_0$ be chosen according to $\mu$. For $n \ge 0$, let  $(\tilde{X}^i_n,1 \le i \le m)$ be independent copies of $X_n$, and set 
\begin{equation}
X_{n+1} = \begin{cases} X_n + D_n &\text{if } X_n = \tilde{X}^i_n \text{ for all } i= 1 \dots m,\\
X_n &\text{if } X_n \neq \tilde{X}^i_n \text{ for some } i \, ,
\end{cases}
\label{eq:GHRWRecursionDefn}
\end{equation}
where we use the convention that $\infty+r=\infty$ and $-\infty+r=-\infty$ for $r \in \mathbb{Z}$. With this definition, 
the random variable $X_n$ has the same distribution as $C^n_\emptyset$ defined as the output of the tree-indexed process. 

In fact, the second formulation can be alternatively defined using a tree-indexed process, but indexed by an $(m+1)$-ary {\em canopy} tree. This is an infinite tree containing a  distinguished one-way infinite path $(v_n,n \ge 0)$, such that after removing the edge $v_nv_{n+1}$, the connected component containing $v_n$ is a complete $(m+1)$-ary tree of depth $n$. 
Adapting the first formulation to the canopy tree in the natural way, the quantity $X_n$ may then be interpreted as the value of the cooperative motion at node $v_n$. 

The entries of the sequence $(X_n,n \ge 0)$ are nicely coupled --- for all $n$ we have $X_{n+1}-X_n \in \{0,D_n\}$ --- so the process may be viewed as a type of random walk with delay. However, the amount of the delay is tied to the law of the process itself, since if $X_n$ finds itself in an unlikely location, then the odds that $\tilde{X}_n^1,\ldots,\tilde{X}_n^m$ are all equal to $X_n$ are low. As such, the position and the rate of motion are highly dependent upon each other, which is the primary challenge in analyzing the process. 

We are unable to characterize all possible asymptotic behaviours of cooperative motion processes, but we do so for a special class of processes we call {\em asymmetric, $q$-lazy cooperative motions}, where the step size distribution corresponds to lazy asymmetric simple random walk. That is, we choose the step sizes $D_n$ to be distributed in the following way. Fix $0 < q \le 1$ and $r \in [0,1]$ with $r \ne 1/2$, then let 
\begin{equation}\label{eq:dn_def}
D_n = 	\begin{cases} 
			1 & \mbox{ with probability } rq\\
			0 & \mbox{ with probability } 1-q\\
			-1 & \mbox{ with probability } (1-r)q. 
		\end{cases}
\end{equation}
We require that $r$ and $q$ are not both equal to one (or else the process is deterministic). 
We write $\acm$ for the law of the process $(X_n, n \geq 0)$ with step sizes $D_n$ as in \eqref{eq:dn_def}, when started from initial distribution $\mu$.

Our main result, Theorem \ref{thm:main2}, shows that after rescaling, all $\acm$ processes are asymptotically Beta-distributed.  
Define for the duration of the paper
\begin{equation*}
\cons = (2r-1)q,
\end{equation*}
and write $\text{sign}(\cons)=\indc_{\{\cons > 0\}}-\indc_{\{\cons < 0\}}$.

\begin{thm}\label{thm:main2} Fix an integer $m \ge 1$, $0 < q \leq 1$, and $0 \leq r \leq 1$, $r \neq 1/2$,  with $q$ and $r$ not both equal to 1. Let $\mu$ be any probability distribution on $\Z$, and let 
$(X_n,n \ge 0)$ be \acm-distributed. Then 
\begin{equation}\label{e.distconv}
\frac{\text{sign}(\cons)}{m+1}\left(\frac{m^m}{|\cons|}\right)^{\frac{1}{m+1}}\cdot\frac{X_n}{n^{1/(m+1)}} \xrightarrow{d} B,
\end{equation}
where $B$ is $\Bet\left(\frac{m+1}{m},1\right)$-distributed. 
\end{thm}
Theorem~\ref{thm:main2} generalizes a result from~\cite{Addario-Berry2019}, which proves a limit theorem for the \acm\ process in the case $m=1,r=1$. We focus on the case $1/2 < r \leq 1$, in which case $\cons > 0$; a simple symmetry of the \acm\ processes then yields the result for the range $0 \leq r < 1/2$. (This is explained in more detail in Remark \ref{r.smallr}).

\subsection{Proof technique}

Let $B$ be $\Bet\left(\frac{m+1}{m},1\right)$-distributed. Our approach to establishing \eqref{e.distconv} is to work directly with the cumulative distribution function (CDF) of the rescaled random variable. In particular, we show that as $n\rightarrow \infty$, the CDF of $n^{-1/(m+1)}X_{n}$ converges to the CDF of $(m+1)\left(\frac{\cons}{m^m}\right)^{1/m+1}B$, which is 
\begin{equation}\label{e.cdfbeta}
F(x)=\begin{cases}
0&\text{if $x\leq 0$}, \\
\frac{m}{\cons^{\frac{1}{m}}(m+1)^{\frac{m+1}{m}}}x^{\frac{m+1}{m}}&\text{if $0\leq x\leq (m+1)\left(\frac{\cons}{m^m}\right)^{1/m+1}$,}\\
1&\text{otherwise.}
\end{cases} 
\end{equation}
Our approach is based on the observation that the CDF of $X_{n}$ in fact satisfies a finite-difference equation which approximates a first-order Hamilton-Jacobi equation. Note that if $F^n_k := \P(X_n < k)$, we have
\begin{align}\label{e.fnk1st}
F^{n+1}_k &= \P(X_n <k-1) + \P(X_n = k-1, X_{n+1} \leq k-1) +\P(X_n = k, X_{n+1} = k-1)\\
&= F^n_{k-1} + \P(X_n=k-1) - \P(X_n =k-1, X_{n+1} = k) +\P(X_n = k, X_{n+1} = k-1)\notag\\
&= F^n_{k-1} + (F^n_k - F^n_{k-1}) - \P(X_n=k-1)\P(D_n=1)\prod_{i=1}^m\P(\tilde{X}^i_n = k-1)\notag\\
&+ \P(X_n=k)\P(D_n=-1)\prod_{i=1}^m\P(\tilde{X}^i_n = k)\notag\\
&= F^n_k - rq\left(F^n_k - F^n_{k-1}\right)^{m+1} + (1-r)q\left(F^n_{k+1} - F^n_{k}\right)^{m+1}\, , \notag
\end{align}
and we may rewrite the final identity as 
\begin{align}
F^{n+1}_k - F^n_k &= - rq(F^n_k - F^n_{k-1})^{m+1} + (1-r)q(F^n_{k+1}-F^n_k)^{m+1} \notag\\
&= -rq|F^n_k - F^n_{k-1}|^{m+1} + (1-r)q|F^n_{k+1}-F^n_k|^{m+1}\, .
\label{eq:CDFRecursion}
\end{align}

The introduction of $|\cdot|$ in \eqref{eq:CDFRecursion} 
is allowed since
$F^n_k \geq F^n_{k-1}$ for all $k$. We write the recursion in the form of \eqref{eq:CDFRecursion} because this makes \eqref{eq:CDFRecursion}  a discrete analogue (or finite difference scheme) of the first-order partial differential equation (PDE)
\begin{equation}\label{e.genhjint}
u_{t}+\cons|u_{x}|^{m+1}=0\quad\text{in $\RR\times (0, \infty)$}.
\end{equation}
In a nutshell, our approach to proving Theorem~\ref{thm:main2} is to exploit this connection, showing that solutions of \eqref{eq:CDFRecursion} closely approximate solutions of \eqref{e.genhjint} after an appropriate rescaling, when $n$ is large. The remainder of the introduction is principally dedicated to elaborating on the details of this approach and the challenges to carrying it out.

Equation \eqref{e.genhjint} is an example of a nonlinear Hamilton-Jacobi equation of the form 
\begin{equation*}
u_{t}+H(u_{x})=0\quad\text{in $\RR\times (0, \infty)$},
\end{equation*}
with the Hamiltonian $H: \RR\rightarrow \RR$ defined by $H(p)=\cons|p|^{m+1}$. For general initial data, \eqref{e.genhjint} fails to have classical, smooth solutions for all time. The theory of viscosity solutions introduced by Crandall and Lions \cite{users,CL}, which are continuous but need not be differentiable, provides a notion of weak solution for such equations. We will hereafter refer to Crandall-Lions viscosity solutions simply as continuous viscosity solutions. We provide an overview of relevant properties of viscosity solutions for Hamilton-Jacobi equations in  Appendix~\ref{app}. 

While continuous viscosity solutions are perhaps the most well-studied notion of weak solution for PDEs such as \eqref{e.genhjint}, our goal is to find a function $u(x,t)$ solving \eqref{e.genhjint}, which is meant to be an $n \to \infty$ analogue of the distribution function
\[
\P\left(\frac{X_{\lfloor tn\rfloor}}{n^{1/(m+1)}}<x\right). 
\]
We note that for any initial distribution $\mu$ with $\mu(\Z)=1$, we have 
\begin{equation}\label{e.heavy}
\P\left(\frac{X_{0}}{n^{1/(m+1)}}<x\right)
\to
\begin{cases}
1 & x > 0,\\
0 & x < 0, \, 
\end{cases}
\end{equation}
as $n\to \infty$, with the behaviour at $x=0$ depending on the distribution $\mu$. This implies that the continuous analogue $u(x,0)$ we seek will necessarily have a discontinuity at $x=0$. Such a discontinuity in the initial condition puts us outside of the framework of continuous viscosity solutions. 

There have been several attempts to define an appropriate notion of discontinuous viscosity solutions (see \cite{chensu} for some references). One notion, introduced by Barron and Jensen \cite{BJ}, is defined for \emph{convex} Hamilton-Jacobi equations. This is our situation; the Hamiltonian $H(p)=\cons|p|^{m+1}$ in \eqref{e.genhjint} is a convex function. (It is for this reason that we introduced absolute values in \eqref{eq:CDFRecursion}.) The Barron-Jensen theory applies exclusively to lower semicontinuous functions, which is why we choose to define $F^n_k=\P(X_n < k)$, instead of the more traditional definition of a CDF given by $\P(X_n \leq k)$. Of course, this makes practically no difference to the probabilistic analysis. 
We will refer to Barron-Jensen viscosity solutions as lsc (lower semicontinuous) viscosity solutions (see the Appendix for more details about the properties of these solutions which we make use of).
Throughout this paper, every continuous (resp. lsc) viscosity solution we consider is in fact the unique continuous (resp. lsc) solution satisfying the PDE in question (see Theorem \ref{t.cvisc} and Theorem \ref{t.bjhl}). Moreover, the two notions coincide for continuous functions. In particular, any lsc viscosity solution which is a continuous function is also a continuous viscosity solution (see Theorem \ref{t.equiv}).

It turns out that the function $F$ introduced in \eqref{e.cdfbeta} is nothing more than $F(x)=u(x,1)$, where $u(x,t)$ is the lsc viscosity solution of the initial value problem 
\begin{equation}\label{e.hjind}
\begin{cases}
u_{t}+\cons|u_{x}|^{m+1}=0&\text{in $\RR\times (0, \infty)$,}\\
u(x,0)=\indc_{\left\{x>0\right\}}&\text{in $\RR$.}
\end{cases}
\end{equation} 

The lsc viscosity solution of \eqref{e.hjind} can be explicitly computed. 
In fact, for future use, we will compute the lsc viscosity solution of the more general PDE 
\begin{equation}\label{e.hjgenind}
\begin{cases}
u^{a,b}_{t}+\cons|u^{a,b}_{x}|^{m+1}=0&\text{in $\RR\times (0, \infty)$,}\\
u^{a,b}(x,0)=a\indc_{\left\{x\leq 0\right\}}+b\indc_{\left\{x>0\right\}}&\text{in $\RR$,}
\end{cases}
\end{equation}
for $0\leq a<b\leq 1$. Since \eqref{e.hjgenind} is a convex Hamilton-Jacobi equation, Theorem \ref{t.bjhl} guarantees that the corresponding lsc viscosity solution is given by the Hopf-Lax formula from control theory, 
\begin{equation}\label{e.hl}
u^{a,b}(x,t)=\inf_{y\in \R} \left\{u^{a,b}(y,0)+tH^{*}\left(\frac{x-y}{t}\right)\right\},
\end{equation}
where $H^{*}$ is the Legendre transform of $H$, defined by 
\begin{equation*}
H^{*}(p)=\sup_{\al\in \RR}\left(\alpha p-H(\al)\right).
\end{equation*}
For the Hamiltonian $H(p)=\cons|p|^{m+1}$, as $H$ is superlinear ($\lim_{p\rightarrow \infty} \frac{H(p)}{|p|}=+\infty$) and $u^{a,b}(x,0)$ is lower semicontinuous, the infimum in \eqref{e.hl} is achieved.  We may thus compute explicitly that for this Hamiltonian, 
\begin{align}
H^{*}(p)&=\sup_{\al\in \R}\left(\alpha p-\cons|\alpha|^{m+1}\right)\notag\\
&=\frac{|p|^{\frac{m+1}{m}}}{(\cons(m+1))^{\frac{1}{m}}}-\cons^{-\frac{1}{m}}\left|\frac{p}{m+1}\right|^{\frac{m+1}{m}}\notag\\
&=\cons^{-\frac{1}{m}}|p|^{\frac{m+1}{m}}\frac{m}{(m+1)^{\frac{m+1}{m}}}.\label{e.hstar}
\end{align}

It follows that the lsc viscosity solution $u^{a,b}$ of \eqref{e.hjgenind} is given by 
\begin{align*}
u^{a,b}(x,t)&=\inf_{y\in \R} \left\{a\indc_{\left\{y\leq 0\right\}}+b\indc_{\left\{y>0\right\}}+tH^{*}\left(\frac{x-y}{t}\right)\right\}\\
&=\inf_{y\in \R} \left\{a\indc_{\left\{y\leq 0\right\}}+b\indc_{\left\{y>0\right\}}+t\cons^{-\frac{1}{m}}\left|\frac{x-y}{t}\right|^{\frac{m+1}{m}}\frac{m}{(m+1)^{\frac{m+1}{m}}}\right\}\\
&=\inf_{y\in \R} \left\{a\indc_{\left\{y\leq 0\right\}}+b\indc_{\left\{y>0\right\}}+\frac{1}{t^{\frac{1}{m}}}\cons^{-\frac{1}{m}}\frac{m}{(m+1)^{\frac{m+1}{m}}}|x-y|^{\frac{m+1}{m}}\right\}.
\end{align*}
A straightforward analysis yields that the preceding infimum is achieved at
\begin{equation*}
y=\begin{cases}
0&\text{if $0\leq \left(\frac{x^{m+1}}{t}\right)^{\frac{1}{m}}\leq (b-a)\cons^{\frac{1}{m}}\frac{(m+1)^{\frac{m+1}{m}}}{m}$,}\\
x&\text{otherwise.}
\end{cases}
\end{equation*}
This implies that 
\begin{equation}\label{eq:scaledICsoln}
u^{a,b}(x,t)=\begin{cases}
a&\text{if $x\leq 0$}, \\
a+\frac{m}{\cons^{\frac{1}{m}}(m+1)^{\frac{m+1}{m}}}\left(\frac{x^{m+1}}{t}\right)^{\frac{1}{m}}&\text{if $0\leq \left(\frac{x^{m+1}}{t}\right)^{\frac{1}{m}}\leq (b-a)\cons^{\frac{1}{m}}\frac{(m+1)^{\frac{m+1}{m}}}{m}$,}\\
b&\text{otherwise.}
\end{cases}
\end{equation}
In the case when $a=0$, and $b=1$ (so for $u$ solving \eqref{e.hjind}), we may rewrite this as 

\begin{equation}
\label{eq:explicitsolution}
u(x,t)=\begin{cases}
0&\text{if $x\leq 0$}, \\
\frac{m}{\cons^{\frac{1}{m}}(m+1)^{\frac{m+1}{m}}}\left(\frac{x^{m+1}}{t}\right)^{\frac{1}{m}}&\text{if $0\leq x\leq (m+1)\left(\frac{\cons t}{m^m}\right)^{1/m+1}$,}\\
1&\text{otherwise,}
\end{cases}
\end{equation}
which agrees with the rescaled Beta CDF given in \eqref{e.cdfbeta} when $t = 1$.
With regards to demonstrating the convergence of the finite difference scheme \eqref{eq:CDFRecursion} to solutions of \eqref{e.genhjint}, we begin by recalling a robust result of Crandall and Lions \cite{Crandall1984}. In \cite{Crandall1984}, the authors identify sufficient conditions for functions defined by finite difference schemes on a space time mesh $\xmesh \mathbb{Z}\times \tmesh \mathbb{N}$ to converge to the continuous viscosity solution of the corresponding Hamilon-Jacobi equation (such as \eqref{e.genhjint}). Their general result is stated as Theorem \ref{t.cl}, below. Upon an appropriate scaling, we may convert \eqref{eq:CDFRecursion} to a finite difference relation on $\xmesh \mathbb{Z}\times \tmesh \mathbb{N}$. Theorem \ref{t.cl} implies that, if \eqref{eq:CDFRecursion} satisfies a monotonicity condition (see Definition \ref{defn:monotonicity}) and $F^{0}_{k}:=u_{0}(k\xmesh)$ is the discretization of a Lipschitz continuous function $u_{0}$ on the mesh $\xmesh\mathbb{Z}$, then for all sufficiently small $\xmesh$, the values $F^N_k$ defined by the finite difference scheme are uniformly close to viscosity solutions $u(k\xmesh,N\tmesh)$ of the PDE with $u(x,0)=u_{0}(x)$, for $N\tmesh$ lying in any compact time interval $[0,T]$. 
The Crandall--Lions theory, however, relies upon the initial data being Lipschitz continuous, as well as using the theory of continuous viscosity solutions. Since we aim to show that the CDFs of the rescaled \acm\ random variables $(X_{n}n^{-1/(m+1)}, n\geq 0)$ converge to the lsc viscosity solution of \eqref{e.hjind}, this precludes a direct application of the results of \cite{Crandall1984} to prove Theorem~\ref{thm:main2}. Furthermore, to the best of our knowledge, no numerical approximation results analogous to those of \cite{Crandall1984} have been proved for lsc viscosity solutions.

Probabilistically, the Lipschitz continuity required by the Crandall-Lions theory is also an issue: it means that the CDF of $X_0/n^{1/(m+1)}$ should be a discretization of a Lipschitz function, with Lipschitz constant independent of $n$; but for a fixed initial distribution, this is impossible (recall \eqref{e.heavy}). To obtain such Lipschitz continuity, the Crandall-Lions theory thus requires the initial condition for the discrete process to depend on the mesh size, which probabilistically translates to requiring the initial distribution of the cooperative motion to depend on the target time $n$ at which we wish to observe the process.
 
In order to make use of the results of \cite{Crandall1984} in
our setting, we use further properties of the probabilistic model in order to demonstrate convergence to the lsc viscosity solution (which corresponds to the Beta-distributed limit in Theorem~\ref{thm:main2}).  In particular, we prove a discrete stochastic monotonicity result, Lemma~\ref{lem:StochasticOrdering}, which allows us to couple the process started from different initial distributions. This coupling is surprisingly delicate; it is not the case that the cooperative motion evolution preserves stochastic ordering for arbitrary initial distributions. However, we prove that it preserves stochastic ordering whenever the initial distribution is not too singular (i.e. when all atoms satisfy a quantitative upper bound, depending on $q$ and $m$). Having established this allows us to use the results of Crandall and Lions \cite{Crandall1984} to prove convergence to the lsc viscosity solution. We stochastically sandwich the evolution started from any initial conditions using Lipschitz-continuous ($n$-dependent) initial conditions, up to an error term which can be made arbitrarily small (after rescaling by $n^{1/(m+1)}$). This allows us to demonstrate the convergence in \eqref{e.distconv} for sufficiently non-singular initial distributions; see Proposition \ref{prop:goodSingIC_conv}. We then conclude by showing that any initial distribution ``relaxes'' to a sufficiently non-singular distribution in a bounded number of steps.

We mention that a recurrence similar to \eqref{eq:CDFRecursion} can be written for the probability mass function $p^n_k = \P(X_n = k)$: 
\begin{equation}\label{eq:pmf_formulation}
p^{n+1}_k - p^{n}_k = -rq\left((p^n_k)^{m+1}-(p^n_{k-1})^{m+1}\right) + (1-r)q\left((p^n_{k+1})^{m+1}-(p^n_{k})^{m+1}\right). 
\end{equation}
This recurrence can be interpreted as a discretization of the scalar conservation law, 
\begin{equation}\label{e.genconsv}
v_{t}=-\cons(v^{m+1})_{x}.
\end{equation}

Indeed, this connection was observed in \cite{Addario-Berry2019} in the special case when $m=1, r = 1$, and the proof in \cite{Addario-Berry2019} of the $m=1, r = 1$ case of Theorem~\ref{thm:main2} relied upon similar numerical PDE $L^{1}$ convergence results for finite difference schemes of scalar conservation laws. In particular, rescaled solutions of the recursion above converge in $L^{1}$ to the unique entropy solutions of \eqref{e.genconsv}. From the theory of PDEs, it is well-known that in the one-dimensional setting, entropy solutions of \eqref{e.genconsv} correspond precisely to derivatives of viscosity solutions of \eqref{e.genhjint}. This motivated our approach of working directly with the CDFs in this paper, and using viscosity solutions methods in this setting. The advantages of working with viscosity solutions include (a) the fact that the solution theory, at least as it relates to such probabilistic models, is better developed for viscosity solutions than for the corresponding entropy solutions, and (b) the fact that working in the ``integrated'' setting gives the solutions greater regularity, which makes the resulting proofs more direct. 

The rest of the paper proceeds as follows. In Section \ref{sect:LipIC}, we review the results of Crandall and Lions \cite{Crandall1984} and use them to demonstrate convergence of CDFs of a rescaled $\acm$ process with a ``diffuse'' initial condition, which approximates a Lipschitz continuous function. In Section \ref{s.goodic}, we show convergence of CDFs of a rescaled $\acm$ process with initial distribution $\mu$ which has no overly large atoms in its support. In Section \ref{sect:GenIC}, we remove this hypothesis on the size of the atoms of $\mu$, and complete the proof of Theorem \ref{thm:main2}. Section \ref{s.generalization} concerns the limitations of the approach taken in this paper, and includes Theorem \ref{thm:no_stoch_mono}, which presents a provable obstacle to applying our methodology to 
establish convergence of cooperative motion-type processes with $|D_n| > 1$. This section also presents Theorem~\ref{thm:main_lattice}, which shows that when the step size is an integer multiple of a Bernoulli random variable, the resulting lattice effects lead to limits which are mixtures of Beta distributions.  
Finally, Appendix \ref{app} provides an overview of continuous and lsc viscosity solutions, and describes several important properties of such solutions that we use throughout the paper. 

\subsection{Notation}

Before proceeding, we introduce some terminological conventions. 
Given a random variable $X$, we define the CDF $F_X:\R \to [0,1]$ of $X$ by $F_X(x) = \P(X < x)$; as mentioned in the introduction, we use this definition rather than the standard $F_X(x)=\P(X \le x)$ to make it easier to appeal to the relevant PDE theory, which has been developed for lower semicontinuous functions. 

We say a function $F:\R\to [0,1]$  is a CDF if it is the CDF of an $\R$-valued random variable, and that $F$ is an extended CDF if it is the CDF of an extended random variable (i.e.\ a random variable taking values in $\R\cup \{\pm \infty\}$). 

For random variables $X,Y$ taking values in $\R\cup \{\pm \infty\}$, we write $X \preceq Y$ and say that $Y$ stochastically dominates $X$ if $\P(X < x) \ge \P(Y < x)$ for all $x \in \R$. 

Let $(X_n,n \ge 0)$ and $(\tilde{X}_n,n \ge 0)$ be ACM$(m,q,r,\mu)$ and ACM$(m,q,r,\tilde{\mu})$-distributed, respectively. 
Suppose that $\tilde{X}_0 \preceq X_0$. 
Then we say that the ACM evolution is {\em stochastically monotone} for $\mu$ and $\tilde{\mu}$ if $\tilde{X}_n \preceq X_n$ for all $n \ge 0$. In other words, the ACM evolution is stochastically monotone for $\mu$ and $\tilde{\mu}$ if it preserves their stochastic ordering in time.

\section{Finite Difference Schemes for Diffuse Initial Conditions}\label{sect:LipIC}

As mentioned in Section \ref{s.intro}, our approach is to interpret CDFs of the discrete random variables $(X_{n}, n\geq 0)$ as solutions of a finite difference scheme. As before, fix $m \in \mathbb{N}$ with $m\geq1$, $q \in (0,1]$, $r \in (1/2,1]$, with $q$ and $r$ not both 1, and a probability distribution $\mu$ supported on $\Z\cup\{-\infty,\infty\}$. Let $(X_n,n \ge 0)$ be $\acm$-distributed, and for $k \in \Z$ write $F^n_k=F^n_k(\mu)=\P(X_n < k)=\mu[-\infty,k)$. (We suppress the dependence on $m,q$ and $r$ as they are fixed throughout, and also suppress the dependence on $\mu$ whenever possible.) Then $(F^n_k)_{k \in \Z, n\in \N}$ is defined by
\begin{equation}\label{eq:Scheme}
\begin{cases}
F^{n+1}_k-F^n_k = -rq\left(F^n_k-F^n_{k-1}\right)^{m+1} + (1-r)q\left(F^n_{k+1}-F^n_{k}\right)^{m+1} & n\ge0,k \in \Z,\\
F^0_k  = \mu[-\infty,k) & k \in \Z. 
\end{cases}
\end{equation}

Since $F^{n}_{k}$ is nondecreasing in $k$ for all $n\in \mathbb{N}$, \eqref{eq:Scheme} can be rewritten as
\begin{equation}\label{e.schemepos}
\begin{cases}
F^{n+1}_k-F^n_k = -rq\left|F^n_k-F^n_{k-1}\right|^{m+1} + (1-r)q\left|F^n_{k+1}-F^n_{k}\right|^{m+1} & n\ge0,k \in \Z, \\
F^0_k  = \mu[-\infty,k) & k \in \Z\, ,
\end{cases}
\end{equation}
and the function defined by \eqref{e.schemepos} is identical to the function defined by \eqref{eq:Scheme}. We will use \eqref{eq:Scheme} and \eqref{e.schemepos} interchangeably, and will also use the fact that $F^{n}_{k}$ is nondecreasing in $k$, for all $n\in \mathbb{N}$, frequently in what follows. 

The main result of this section is the following proposition, which states that solutions of the recurrence relation from \eqref{e.schemepos}, with nondecreasing, Lipschitz initial data converge to solutions of the appropriate Hamilton-Jacobi equation. 
\begin{prop}\label{p.intlip}
Let $u_0$ be a Lipschitz-continuous extended CDF with Lipschitz constant $K$. Fix $N \in \N$ and define a probability distribution $\mu_N$ on $\Z \cup \{-\infty,\infty\}$ by  
\[
\mu_N[-\infty,k):=u_0(kN^{-1/(m+1)}) 
\]
for $k \in \Z$. Let $(X_n,n \ge 0)$ be ACM$(m, q, r, \mu_{N})$-distributed, and let $F^n_k = F^n_k(\mu_N) = \P(X_n < k)$. Finally, fix $T>0$. Then there exist $N_0=N_0(m,q,K)$ and $c=c(K,m, \cons, T)$ such that if $N \ge N_0$, 
\begin{equation}\label{e.invp}
\sup_{0 \le t \le T} 
\sup_{k \in \Z} 
\left|
F^{\lfloor Nt\rfloor}_k - u\Big(\frac{k}{N^{1/(m+1)}},t\Big)
\right| 
\le \frac{c}{N^{1/2}}\, ,
\end{equation}
where $u$ is the continuous viscosity solution of 
\begin{equation}\label{e.genhj}
\begin{cases}
u_{t}+\cons|u_{x}|^{m+1}=0&\text{in $\RR\times (0, \infty)$,}\\
u(x,0)=u_{0}(x)&\text{in $\RR$.}
\end{cases}
\end{equation}
It follows that $u$ is an extended CDF and that 
\begin{equation}\label{e.invp1}
\sup_{x \in \R} 
\left|
\P\Big(\frac{X_N}{N^{1/(m+1)}}< x \Big)-u(x,1)
\right| \le 
\frac{c}{N^{1/2}}\, .
\end{equation}
\end{prop}
In order to prove this proposition, we require the framework of monotone finite different schemes for Hamilton-Jacobi equations. We next introduce this framework, and relate it to the evolution of the CDF of cooperative motion. 

We may imagine numerically approximating the solution of a Hamilton-Jacobi equation of the form
\begin{equation}\label{e.hj_template}
\begin{cases}
u_{t}+H(u_{x})=0&\text{in $\RR\times (0, \infty)$,}\\
u(x,0)=u_{0}(x)&\text{in $\RR$}
\end{cases}
\end{equation}
as follows. Fix temporal and spatial mesh sizes ($\tmesh$ and $\xmesh$, respectively). 
Set $U^0_k = u_0(k\xmesh)$ for $k \in \Z$, and 
for $n \ge 0$ define $U^{n+1}_k$ by  
\begin{equation}\label{e.genmon}
U^{n+1}_{k}=G(U^n_{k+1},U^{n}_{k}, U^{n}_{k-1}):=U^{n}_{k}-\tmesh \hfunc\left(\frac{U^{n}_{k+1}-U^{n}_{k}}{\xmesh},\frac{U^{n}_{k}-U^{n}_{k-1}}{\xmesh}\right).
\end{equation}
Here $\hfunc:\R\times\R \to \R$ is a function to be chosen, which is meant to act as an approximation of $H$. We may write the function $G$ in \eqref{e.genmon} as 
$G(x,y,z) = G^{\Delta}(x,y,z)= y - \tmesh \hfunc(\tfrac{x-y}{\xmesh},\tfrac{y-z}{\xmesh})$, 
where $\Delta=(\xmesh,\tmesh)$.

We now use \eqref{e.genmon} to define a rescaled field of values 
\[
u^\Delta:  \xmesh\Z \times\tmesh\N\to \R
\] 
by setting $u^\Delta(k\xmesh, n\tmesh):= U^n_k$.With this definition, \eqref{e.genmon} is equivalent to the statement that 
\begin{multline*}
\frac{u^{\Delta}(k\xmesh, n\tmesh+\tmesh) - u^{\Delta}(k\xmesh,n\tmesh)}{\tmesh}
\\+
\hfunc\Big(\frac{u^{\Delta}(k\xmesh+\xmesh, n\tmesh)-u^{\Delta}(k\xmesh, n\tmesh)}{\xmesh},\frac{u^{\Delta}(k\xmesh, n\tmesh)-u^{\Delta}(k\xmesh-\xmesh, n\tmesh)}{\xmesh}\Big)=0.
\end{multline*}
Each of the arguments of $g$ looks like an approximation of $u^{\Delta}_x$, so this in some sense looks formally like a discretization of  \eqref{e.hj_template} on the space-time mesh $\xmesh \Z\times \tmesh \N$. 
We refer to \eqref{e.genmon} as a {\em finite difference scheme} for the initial value problem \eqref{e.hj_template}. 

It turns out that, under suitable regularity assumptions on the initial condition $u_0$ and the Hamiltonian $H$, the sufficient conditions on \eqref{e.genmon} for 
$u^\Delta$, or equivalently 
$(U^n_k)_{k \in \Z,n \in \N}$, 
to well-approximate $u$ as $\tmesh$ and $\xmesh \to 0$ are \emph{consistency} and {\em monotonicity}. The consistency condition is simply that $\hfunc(p,p)=H(p)$ for all $p\in\mathbb{R}$. 
The monotonicity condition is as follows. 

\begin{define}\label{defn:monotonicity}
We say recurrence of the form \eqref{e.genmon} is monotone on an interval $[\lambda,\Lambda]\subseteq \mathbb{R}$ if $G(U^n_{k+1},U^{n}_{k}, U^{n}_{k-1})$ is a nondecreasing function of each argument so long as for all $k$, 
\begin{equation}\label{e.disclip}
\lambda\leq (\xmesh)^{-1}\left(U^{n}_{k+1}-U^{n}_{k}\right)\leq \Lambda.
\end{equation} 
\end{define}
We now state the main result from \cite{Crandall1984}, specialized to the one-dimensional setting of the current paper, on the quality of approximation provided by monotone finite difference schemes for Hamilton-Jacobi equations. 
\begin{thm}\label{t.cl}[Theorem 1, \cite{Crandall1984}]
Let $u:\RR\times (0, \infty)\rightarrow \RR$ be the continuous viscosity solution of 
\begin{equation}\label{e.ivp}
\begin{cases}
u_{t}+H(u_{x})=0&\text{in $\RR\times (0, \infty)$,}\\
u(x,0)=u_{0}(x)&\text{in $\RR$,}
\end{cases}
\end{equation}
where $H:\RR\rightarrow \RR$ is continuous and $u_{0}$ is bounded and Lipschitz continuous with Lipschitz constant $K$.  Fix $\xmesh > 0$ and $\tmesh > 0$, let $U^{0}_{k}:=u_{0}(k\xmesh)$, and define $U^{n}_{k}$ by a general scheme of the form \eqref{e.genmon}. 

If \eqref{e.genmon} is consistent and monotone on $[-(K+1),K+1]$, then for any $T >0,$ there exists $c$, depending on $\sup |u_{0}|, K, H$, and $T$ so that
\begin{equation}
\sup_{n \in \N,n\tmesh\in[0,T]} \sup_{k \in \Z}|U^{n}_{k}-u(k\xmesh, n\tmesh)|\leq c\sqrt{\tmesh}. 
\end{equation}
\end{thm}

Before connecting Theorem~\ref{t.cl} to cooperative motion, it is instructive to further discuss the meaning and value of monotonicity in this setting. (The following discussion is inspired by the proof of \cite[Proposition~3.1]{Crandall1984}.)
Fix $K > 0$ and two sets of initial conditions $(U^0_k)_{k \in \Z}$ and $(\tilde{U}^0_k)_{k \in \Z}$ with $U^0_k \le \tilde{U}^0_k$, then set $U^{n+1}_k = G(U^n_{k+1},U^n_k,U^n_{k-1})$ and $\tilde{U}^{n+1}_k = G(\tilde{U}^n_{k+1},\tilde{U}^n_k,\tilde{U}^n_{k-1})$ for $n \ge 0$ as in \eqref{e.genmon}.

Suppose that $G$ is monotone on $[-K,K]$, and that 
\begin{equation}\label{e.locmon}
\frac{|U^0_k-U^0_{k-1}|}{\xmesh}\le K
\quad \mbox{and} \quad
\frac{|\tilde{U}^0_k-\tilde{U}^0_{k-1}|}{\xmesh}\le K
\end{equation}
for all $k$. Then monotonicity implies that 
\begin{equation}\label{e.monoprop}
U^1_k =
G(U^0_{k+1},U^0_k,U^0_{k-1})\le 
G(\tilde{U}^0_{k+1},\tilde{U}^0_k,\tilde{U}^0_{k-1}) =\tilde{U}^1_k \, .
\end{equation}
Next, let $(W^{0}_{k})_{k\in \mathbb{Z}}$ be any initial condition with $\sup_k \xmesh^{-1} |W^0_k-W^0_{k-1}| \le K$, and set $W^{n+1}_{k}=G(W^n_{k+1},W^{n}_{k}, W^{n}_{k-1})$ for $n \ge 0$ and $k \in \Z$. Write $\lambda=\sup_k |W^0_k-U^0_k|$. Let $V^0_k=U^0_k+\lambda$, and set $V^1_k = G(V^0_{k+1},V^0_k,V^0_{k-1})=U^1_k+\lambda$. By the choice of $\lambda$, we have $W^{0}_{k}\leq V^{0}_{k}$. Then monotonicity gives that 
\[
W^1_k \le V^1_k=U^1_k+\lambda\, ,
\]
and a symmetric argument gives that $W^1_k \geq U^1_k - \lambda$, so 
\[
\sup_k |W^1_k - U^1_k|\le \lambda. 
\]
We apply this with the specific choice of initial condition $W^0_k = U^0_{k-1}$. Since $W^{1}_{k}=U^{1}_{k-1}$, the preceding bound gives
\[
\sup_k|U^1_{k}-U^1_{k-1}|=\sup_k|U^1_k-W^1_{k}|
 \le \lambda 
= \sup_k |U^0_k-W^0_k|
=\sup_k|U^0_{k} - U^0_{k-1}| \le K \xmesh.  
\]
A similar analysis allows us to conclude that 
\begin{equation*}
\sup_k|\tilde{U}^1_{k}-\tilde{U}^1_{k-1}|\leq  K \xmesh. 
\end{equation*}
By the two preceding bounds and \eqref{e.monoprop}, it follows by induction that $U^n_k \le \tilde{U}^n_k$ for all $n \in \N$ and $k \in \Z$ and that  $\sup_k|U^n_k-U^n_{k-1}| \le K\xmesh$ for all $n$. In short, equation \eqref{e.disclip}, which can be viewed as a type of discrete Lipschitz bound on $U^{n}_{k}$, allows one to show that an order relation between two initial conditions persists for all positive times. 
\begin{remark}\label{r.mon_new}
Whenever the initial condition $(U^0_k)_{k \in \Z}$ is nondecreasing in $k$, the above argument shows that if $G$ is monotone on $[0,K]$ and $\sup_k(U^0_k-U^0_{k-1}) \le K\xmesh$, then $(U^n_k)_{k \in \Z}$ is nondecreasing in $k$ and $\sup_k(U^n_k-U^n_{k-1}) \le K\xmesh$ for all $n \in \N$. 
It follows from this that if $u_0$ is nondecreasing, then in order to verify the condition of Theorem~\ref{t.cl} one need only check that \eqref{e.genmon} is monotone on $[0,K+1]$. 
\end{remark}

We now specialize the above discussion to the specific setting of our paper, so again let $F^n_k=\P(X_n<k)$ where $(X_n, {n \ge 0})$ is $\acm$-distributed. Given spatial and temporal mesh sizes ($\xmesh$ and $\tmesh$, respectively), we may use the field of values $(F^n_k)_{k \in \Z,n \in \mathbb{N}}$ to define a rescaled field $f=f_{\tmesh,\xmesh}:  \xmesh\mathbb{Z}\times \tmesh\mathbb{N} \to \R$ by setting $f(k\xmesh, n\tmesh):= F^n_k$.

In order to identify an appropriate scaling relationship between $\xmesh$ and $\tmesh$, we seek a continuous space-time scaling which preserves the PDE. In particular, if $u$ solves \eqref{e.hjind}, then for any $\rho\in \RR$, $u_{\rho}(x,t):=u(\rho x, \rho^{m+1}t)$ also solves \eqref{e.hjind}. This suggests that the temporal and spatial mesh sizes should satisfy the relation 
\begin{equation}\label{eq:SchemeStepSizes}
(\xmesh)^{m+1}=\tmesh.
\end{equation} 
With this scaling relation, 
we may rewrite \eqref{eq:Scheme} as 
\begin{equation}\label{eq:Scheme2}
F^{n+1}_{k}=F^{n}_{k}-\tmesh\left[rq\left(\frac{F^{n}_{k}-F^{n}_{k-1}}{\xmesh}\right)^{m+1}-(1-r)q\left(\frac{F^{n}_{k+1}-F^{n}_{k}}{\xmesh}\right)^{m+1}\right], 
\end{equation}
which, since $F^n_k$ is nondecreasing in $k$, we may re-express as  
\begin{equation}\label{eq:Scheme3}
F^{n+1}_{k}=F^{n}_{k}-\tmesh\left[rq\left|\frac{F^{n}_{k}-F^{n}_{k-1}}{\xmesh}\right|^{m+1}-(1-r)q\left|\frac{F^{n}_{k+1}-F^{n}_{k}}{\xmesh}\right|^{m+1}\right]. 
\end{equation}
This equation has precisely the form of \eqref{e.genmon} with 
$$G(x,y,z)=y-\tmesh[-(1-r)q|\tfrac{x-y}{\xmesh}|^{m+1}+rq|\tfrac{y-z}{\xmesh}|^{m+1}]=y+(1-r)q|x-y|^{m+1}-rq|y-z|^{m+1},$$ the second equality holding due to (\ref{eq:SchemeStepSizes}).

Now fix probability distributions $\mu,\tilde{\mu}$ on $\Z\cup \{\pm \infty\}$ with $\tilde{\mu}\preceq\mu$, let $(X_n, n \ge 0)$ and $(\tilde{X}_n, n \ge 0)$ be ACM$(m,q,r,\mu)$ and ACM$(m,q,r,\tilde{\mu})$-distributed, respectively, and set $F^n_k=\P(X_n < k)$ and $\tilde{F}^n_k=\P(\tilde{X}_n < k)$, so $(F^n_k)_{k\in \Z, n\in \N}$ and $(\tilde{F}^n_k)_{k\in \Z, n\in \N}$ both satisfy \eqref{eq:Scheme3} but with different initial conditions. The fact that $\tilde{\mu}\preceq\mu$ means that $F^0_k \le \tilde{F}^0_k$. 

For a given $\Lambda > 0$, if $G$ is monotone on $[0,\Lambda]$ and $\sup_{k \in \Z} (F^0_k-F^0_{k-1})\le \Lambda$ and $\sup_{k \in \Z}(\tilde{F}^0_k-\tilde{F}^0_{k-1})\le \Lambda$, 
then \eqref{e.monoprop}
gives that $F^1_k \le \tilde{F}^1_k$, and inductively that $F^n_k \le \tilde{F}^n_k$ for all $n$. In other words, we can think of monotonicity as a sufficient condition which guarantees that the cooperative motion will preserve stochastic ordering in time. In Section \ref{s.goodic}, we will use a variation of this approach to identify the value of $\Lambda$, and thereby a sufficient condition, which guarantees stochastic monotonicity.
\begin{proof}[Proof of Proposition~\ref{p.intlip}] 
Let $N_0:=\left([q(m+1)]^{1/m}(K+1)\right)^{(m+1)}$ and fix $N\geq N_0$. We choose $\xmesh=N^{-1/(m+1)}$, and $\tmesh=N^{-1}$, so that $\xmesh$ and $\tmesh$ satisfy \eqref{eq:SchemeStepSizes}. This implies that $F^{(\cdot)}_{k}(\mu_{N})$ is defined by \eqref{eq:Scheme3}. The proof relies upon verifying the hypotheses of Theorem \ref{t.cl} for $U^n_k=F^n_k$.

Consistency is easily verified in our setting.
The Hamilton-Jacobi equation in question is \eqref{e.genhj}, for which $H(p) = \cons p^{m+1}$. Moreover, 
we have $\hfunc(x,y) = -(1-r)qx^{m+1} + rqy^{m+1}$, so $\hfunc(p,p) = (2r-1)qp^{m+1} = \cons p^{m+1} = H(p)$, as desired.

Now we check monotonicity. As $u_{0}$ is nondecreasing, by Remark \ref{r.mon_new}, we only need to verify that  \eqref{eq:Scheme3} or, equivalently, \eqref{eq:Scheme2} is monotone in $[0, K+1]$. 

To verify monotonicity of \eqref{eq:Scheme2}, we differentiate 
\begin{equation*}
G(F^n_{k+1},F^{n}_{k}, F^{n}_{k-1})=F^{n}_{k}-rq\tmesh\left(\frac{F^{n}_{k}-F^{n}_{k-1}}{\xmesh}\right)^{m+1}+(1-r)q\tmesh\left(\frac{F^{n}_{k+1}-F^{n}_{k}}{\xmesh}\right)^{m+1}
\end{equation*}
in each argument, in the region $0\leq (\xmesh)^{-1}(F^{n}_{k+1}-F^{n}_{k}),(\xmesh)^{-1}(F^{n}_{k}-F^{n}_{k-1})\leq K+1$. Differentiating the right hand side with respect to $F^{n}_{k}$, we have 
\begin{align*}
&1-rq(m+1)\frac{\tmesh}{\xmesh}\left(\frac{F^{n}_{k}-F^{n}_{k-1}}{\xmesh}\right)^{m}-(1-r)q(m+1)\frac{\tmesh}{\xmesh}\left(\frac{F^{n}_{k+1}-F^{n}_{k}}{\xmesh}\right)^{m}\\
&\geq 1-rq(m+1)\frac{\tmesh}{\xmesh}(K+1)^{m}-(1-r)q(m+1)\frac{\tmesh}{\xmesh}(K+1)^{m}\\
&=1-q(m+1)(\xmesh)^{m}(K+1)^{m}. 
\end{align*}
As $N\geq N_0$, we have that 
\begin{equation*}
(\xmesh)^{m}=N^{-m/(m+1)}\leq [q(m+1)]^{-1}(K+1)^{-m},
\end{equation*}
which implies that $G(F^n_{k+1},\cdot, F^{n}_{k-1})$ is nondecreasing. 
Similar computations show that both $G(\cdot,F^{n}_{k}, F^n_{k-1})$ and $G(F^n_{k+1},F^{n}_{k}, \cdot)$ are nondecreasing. 
This implies that \eqref{eq:Scheme2} is monotone on $[0, K+1]$, so by Theorem \ref{t.cl}, we then have that for $N\geq N_0$, for any $T>0$, 
there is $c=c(K,m,q,T)$ such that 
\begin{equation*}
\sup_{0 \le j/N \le T} \sup_{k \in \Z} \left|F^{j}_{k}-u\left(\frac{k}{N^{1/(m+1)}}, \frac{j}{N}\right)\right|\leq cN^{-\frac{1}{2}};
\end{equation*}
recall that $\tmesh = N^{-1}$, so $j/N=j\tmesh$. 
We may rewrite this bound as 
\begin{equation*}
\sup_{0 \le t \leq T}\sup_{k \in \Z} \left|F^{\lfloor Nt\rfloor}_{k}-u\left(\frac{k}{N^{1/(m+1)}}, \frac{1}{N}\lfloor Nt \rfloor \right)\right|\leq cN^{-\frac{1}{2}}. 
\end{equation*}

By Proposition \ref{p.regularity}, the continuous viscosity solution $u$ solving \eqref{e.genhj} is globally Lipschitz continuous in space and time. Therefore, 
\begin{align*}
&\sup_{0 \le t \leq T}\sup_{k \in \Z} \left|F^{\lfloor Nt\rfloor}_{k}-u\left(\frac{k}{N^{1/(m+1)}}, t \right)\right|\\
&\leq  \sup_{0 \le t \leq T}\sup_{k \in \Z} \left|F^{\lfloor Nt\rfloor}_{k}-u\left(\frac{k}{N^{1/(m+1)}}, \frac{1}{N}\lfloor Nt \rfloor \right)\right|\\
& 
\quad +
\sup_{0 \le t \leq T}\sup_{k \in \Z} \left| u\left(\frac{k}{N^{1/(m+1)}}, \frac{1}{N}\lfloor Nt \rfloor \right)-u\left(\frac{k}{N^{1/(m+1)}}, t \right)\right|\\
&\leq cN^{-\frac{1}{2}}+ C\sup_{0\leq t\leq T}\left|\frac{1}{N}\lfloor Nt \rfloor-t\right|\leq \tilde cN^{-\frac{1}{2}},
\end{align*}
and this yields \eqref{e.invp}; equation~\eqref{e.invp1} follows as it is simply a restating of \eqref{e.invp} in the special case when $t=1$. Finally, \eqref{e.invp1} gives that $u(x,1)$ is the pointwise limit of an extended CDF, so $u(x,1)$ is itself an extended CDF. 
\end{proof}

\section{``Good'' singular initial conditions}\label{s.goodic}

The convergence results of the previous section require that the finite difference scheme $(F^n_k)_{k\in \Z, n\in \N}$ begins with an initial condition $\mu_N$ which is a discretization of a Lipschitz function at scale $\xmesh$ (depending on $N$). In this section, we build on those convergence results to prove distributional limit theorems for certain fixed (rather than varying in $N$) initial conditions. 
Let 
\begin{equation}\label{e.p*def}
p^* = \left(\frac{1}{q(m+1)}\right)^{1/m}.
\end{equation}
Note that $p^* > 1/2$ whenever $m$ and $q$ are not both equal to 1, and $p^*=1/2$ if and only if $m = q = 1$. 
We say that an extended probability distribution $\mu$ is {\em $p^*$-bounded} if 
\begin{equation*}
\sup_{x \in \Z \cup \{-\infty,\infty\}} \mu(\{x\}) \leq p^*.
\end{equation*}
The goal of this section is to prove the following proposition, which essentially states that Theorem~\ref{thm:main2} holds for $p^*$-bounded initial conditions. 
\begin{prop}\label{prop:goodSingIC_conv}
Let ($X_n$, $n \geq 0$) be \acm-distributed with $\mu$ a probability distribution on $\Z$. If $\mu$ is $p^*$-bounded, then 
$$\lim_{n\to\infty} \P\left(\frac{X_n}{n^{1/(m+1)}} <x\right) = u(x, 1)$$
uniformly in $x$, where $u(x,t)$ is given by \eqref{eq:explicitsolution}.
\end{prop}
The proof of Proposition \ref{prop:goodSingIC_conv} relies on comparison between the ACM evolution with $p^*$-bounded initial conditions to ACM evolutions with Lipschitz continuous initial conditions. 
To establish the possibility of such comparisons, we prove that the ACM evolution is stochastically monotone on a much broader class of initial conditions than what is covered by 
Proposition~\ref{p.intlip}. (It may be useful to revisit the discussion preceding the proof of Proposition~\ref{p.intlip} at this point.) 
We first show for the class of $p^*$-bounded distributions, stochastic ordering is preserved in one time-step of the ACM evolution. We then show that the CDFs at future time steps remain in the family of $p^*$-bounded distributions. This is exactly the content of the next two lemmas: 

\begin{lemma}\label{lem:StochasticOrdering}
Let $\mu_X$ and $\mu_Y$ be $p^*$-bounded probability distributions on $\Z \cup \{-\infty, \infty\}$, and let $(X_n, n\geq 0)$ be $\mathrm{ACM}(m,q,r,\mu_X)$-distributed and $(Y_n,n\geq 0)$ be $\mathrm{ACM}(m,q,r,\mu_Y)$-distributed. If $Y_0 \preceq X_0$ then $Y_1 \preceq X_1$. 
\end{lemma}

\begin{lemma}\label{lem:discrete1stepdynamics}
Let $(X_n, n\geq 0)$ be $\acm$-distributed and define $P(X_n = k) = p^n_k$. Then for all $k \in \Z$

\begin{equation}\label{eq:1stepinequality}
p^{n+1}_k \leq \max(p^n_{k-1}, p^n_k, p^n_{k+1})
\end{equation}
\end{lemma}

\begin{proof}[Proof of Lemma \ref{lem:StochasticOrdering}]
Recall that $G(x,y,z) = y+(1-r)q|x-y|^{m+1}-rq|y-z|^{m+1}$. Note that for $0 \le z \le y \leq x$, $G(x,y,z)$ is nondecreasing in $x$ and $z$, and for $0 \le z \le y \le x$, $G$ is nondecreasing in $y$ provided that  
\[
1-(m+1)rq(y-z)^m -(m+1)(1-r)q(x-y)^m \ge 0\, ,
\]
which is true whenever $y-z,x-y \le p^*$. Therefore, $G$ is monotone on $[0,p^*]$. 

Now write $F^0_k = \P(X_0 < k) = \mu_X[-\infty,k)$ and $\tilde{F}^0_k =\P(Y_0 < k) = \mu_Y[-\infty,k)$. 
Since $Y_0 \preceq X_0$, we have that for all $k \in \Z$, $F^0_k \le \tilde{F}^0_k$, and moreover, $F^0_k - F^0_{k-1} \le p^*$ and $\tilde{F}^0_k - \tilde{F}^0_{k-1} \le p^*$. Since $G$ is monotone on $[0,p^*]$ it follows that 
\[
\P(X_1 < k) = F^1_k = G(F^0_{k+1},F^0_k,F^0_{k-1}) \le 
G(\tilde{F}^0_{k+1}\tilde{F}^0_k,\tilde{F}^0_{k-1})=\tilde{F}^1_k
=\P(Y_1 < k)\, ,
\]
so $Y_1 \preceq X_1$, as required. 
\end{proof}
The choice of $p^*$ as an upper bound for a single site probability in Lemma \ref{lem:StochasticOrdering} is sufficient, but since $p^*$ is independent of $r$, one may ask whether $p^*$ is the tightest upper bound which gives stochastic ordering of the process. In fact, it is not in general: a little algebra shows that the necessary upper bound is the value $M \geq 1/2$ such that 
\[
\frac{1}{q(m+1)} = rM^{m} + (1-r)(1-M)^{m}.
\]
From this definition and the fact that $M \geq 1-M$, it follows that $M \geq p^*$. There are cases when $M=p^*$, in particular when $r=1$, or when $m=q=1$, but this equality does not hold in general. For example, if $m = 1, q = 2/3, r = 2/3$, we see that $p^* = 3/4$ while $M = 5/4$. 

We next introduce an additional technical lemma needed in the proof of Lemma~\ref{lem:discrete1stepdynamics}. Write $f(p):=p-qp^{m+1}$; we will use that $f$ is  increasing on $[0,p^*)$ and decreasing on $(p^*,1]$.

\begin{lemma}\label{lem:h(x,y)facts}
Let $\gfunc(x,y) = f(x) - f(y)$ with $f(x)=x-qx^{m+1}$. If $0 < q \leq 1$, then $\gfunc(a,b)\geq 0$ whenever $a \geq b \geq 0$ and $a + b \leq 1$. Under the additional constraint $a > p^*$, we have
\begin{equation}\label{eq:h(x,y)bdd>0}
\gfunc(a,b) \geq \min(\gfunc(p^*,1-p^*),1-q) \geq 0.
\end{equation}
Moreover, when $q < 1$, $\min(\gfunc(p^*,1-p^*),1-q) > 0$. 
\end{lemma}

\begin{proof}
First, we note that
\begin{align*}
\frac{\partial}{\partial	x} \gfunc(x,y) &= 1-q(m+1)x^m,\\
\frac{\partial}{\partial	y} \gfunc(x,y) &= -1+q(m+1)y^m.
\end{align*}
We are concerned with the behavior of the function $h$ in the regions $A$ and $B$ shown in Figure~\ref{fig:g_zones}. Formally, if $C=\{(x,y):0 \le y \le x,y \le 1-x\}$, then $A=\{(x,y) \in C: x \le p^*\}$ and $B=\{(x,y) \in C: x > p^*\}$.

\begin{figure}[htb]
\begin{tikzpicture}

\begin{axis}[
xmin = 0, 
xmax=1, 
ymin=0, 
ymax=1, 
xtick = {0,.5,.75,1}, 
ytick = {0,.5,.75,1}, 
xticklabels = {$0$, $1/2$, $p^*$, $1$}, 
yticklabels = {$0$, $1/2$, $p^*$, $1$}]

\addplot[name path = V2, dashed] coordinates {(.75,0)(.75,1)};
\addplot[name path = H2, dashed] coordinates {(0,.75)(1,.75)};
\addplot[name path = D1, ultra thick, color = red] {x};
\addplot[name path = D2, ultra thick, color = blue] {1-x};
\path [name path = axis] (axis cs:0,0) -- (axis cs: 1,0);

\addplot[fill = red, opacity=0.4] fill between [of = D1 and axis, soft clip = {domain=0:0.5}];
\addplot[fill = red, opacity=0.4] fill between [of = D2 and axis, soft clip = {domain=0.5:0.75}];

\node at (axis cs: .5,.2) {$A$};

\addplot[fill = teal, opacity=0.4] fill between [of = D2 and axis, soft clip = {domain = .75:1}];

\node at (axis cs: .82,.1) {$B$};

\end{axis}
\end{tikzpicture}
\caption{Lemma~\ref{lem:h(x,y)facts} states that $\gfunc(a,b)$ is positive for $(a,b) \in A$ and is greater than $\min(\gfunc(p^*,1-p^*),1-q)$ for $(a,b)\in B$.}
\label{fig:g_zones}
\end{figure}
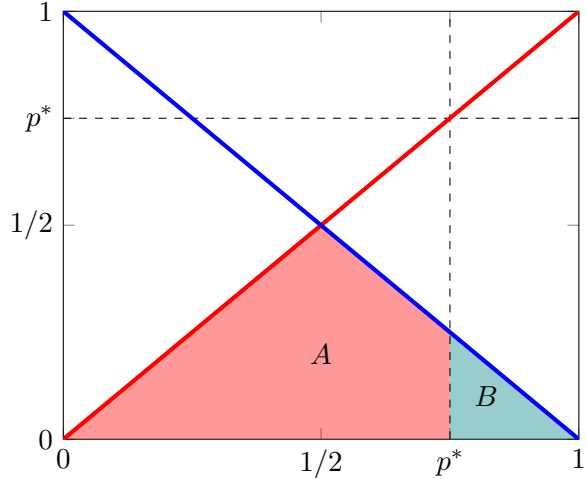

If $(x,y) \in A$ then since $f$ is increasing on $[0,p^*)$ and in this region $0 \le y \le x \le p^*$, it follows that $\gfunc(x,y)=f(x)-f(y) \geq 0$. 

To determine the behavior in region $B$, notice that in this region, $\frac{\partial h}{\partial x} < 0$ and $\frac{\partial h}{\partial y} < 0$. Therefore, 
\[
\inf_{(x,y)\in B} \gfunc(x,y)=\inf_{x \in (p^*,1]} \gfunc(x,1-x).
\]
Moreover, since $x > 1/2$ and 
\[\frac{\partial^2}{\partial x^2} \gfunc(x,1-x) = q(m+1)m((1-x)^{m-1} - x^{m-1})\, ,
\]
the difference $(1-x)^{m-1}-x^{m-1}$ is strictly negative. This imples that $\gfunc(x,1-x)$ is strictly concave for $x \in (1/2,1)$. Since $p^* \geq 1/2$, we thus have 
\begin{equation}\label{e.inf_reduction}
\inf_{(x,y)\in B} \gfunc(x,y)=\min_{x \in \{p^*,1\}} \gfunc(x,1-x)\geq 
\min_{x \in \{1/2,1\}} \gfunc(x,1-x).
\end{equation}
Since $\gfunc(1/2,1/2)=0$ and $\gfunc(1,0)=1-q \geq 0$, it follows that $\gfunc(x,1-x) \geq 0$ for all $x \in [1/2,1]$. 

From the relations in \eqref{e.inf_reduction} and the fact that $\gfunc(x,1-x)$ is strictly concave on $(1/2,1)$, we see that when $x > p^*$
\[
\gfunc(x,y) \geq \min(\gfunc(p^*,1-p^*),1-q) \geq 0,
\]
as desired. Notice that when $q < 1$, we have that $p^* > 1/2$, so the strict concavity of $\gfunc(x,1-x)$ tells us that in this case $\gfunc(p^*,1-p^*)>0$. This fact proves the last statement of the lemma; when $q < 1$, we have that $\min(\gfunc(p^*,1-p^*),1-q)>0$.
\end{proof}

Equipped with this technical lemma, we can now prove Lemma \ref{lem:discrete1stepdynamics}. 

\begin{proof}[Proof of Lemma \ref{lem:discrete1stepdynamics}]
We prove the lemma in cases, according to which site has the maximum value. In each case, we will show that $\max(p^n_{k-1}, p^n_k, p^n_{k+1}) - p^{n+1}_k \geq 0$.

Recall the recursion equation for $p^{n+1}_k$:
\begin{equation*}
p^{n+1}_k = p^{n}_k + q\left[r(p^n_k)^{m+1}-(p^n_{k-1})^{m+1} + (1-r)(p^n_{k+1})^{m+1})^{m+1}\right]. 
\end{equation*} 
First, when $p^n_k = \max(p^n_{k-1}, p^n_k, p^n_{k+1})$, we have  
\begin{align*}
p^n_k - p^{n+1}_k &= -q\left[r(p^n_{k-1})^{m+1}-(p^n_k)^{m+1}+(1-r)(p^n_{k+1})^{m+1}\right]\\
&\geq -q\left[r\left(p^n_k\right)^{m+1}-(p^n_k)^{m+1}+(1-r)\left(p^n_k\right)^{m+1}\right]\\
&= 0,
\end{align*}
where we have used the fact that the function $x^{m+1}$ is increasing for positive $x$ values. 

Second, we consider the case when $p^n_{k+1} = \max(p^n_{k-1}, p^n_k, p^n_{k+1})$. 
In this case, we have  
\begin{align*}
p^n_{k+1} - p^{n+1}_k &= p^n_{k+1}-\left[p^n_k+ qr(p^n_{k-1})^{m+1}-q(p^n_k)^{m+1}+(1-r)q(p^n_{k+1})^{m+1}\right]\\
&\geq p^n_{k+1}-\left[p^n_k+ qr(p^n_{k+1})^{m+1}-q(p^n_k)^{m+1}+(1-r)q(p^n_{k+1})^{m+1}\right]\\
&= \left[p^n_{k+1}-q(p^n_{k+1})^{m+1}\right]-\left[p^n_{k}-q(p^n_{k})^{m+1}\right]\\
&= \gfunc(p^n_{k+1}, p^n_k),
\end{align*}
where $\gfunc(x,y)$ is defined as in Lemma \ref{lem:h(x,y)facts}. Because $p^n_{k+1} \geq p^n_k$ and $p^n_{k+1}+p^n_k \leq 1$, we can apply Lemma \ref{lem:h(x,y)facts} to say that $p^n_{k+1}-p^{n+1}_k\geq 0$. 

Finally, the case when $p^n_{k-1} = \max(p^n_{k-1}, p^n_k, p^n_{k+1})$ can be proved by a symmetric argument to that of the second case. 
\end{proof}

\begin{remark}\label{r.newmon}
Combining Lemma \ref{lem:StochasticOrdering} and Lemma \ref{lem:discrete1stepdynamics}, we are able to identify a precise value of $\Lambda$ which guarantees stochastic monotonicity, as discussed just above the proof of Proposition \ref{p.intlip}.  In particular,  if $F^{0}_{k}$ and $\tilde{F}^{0}_{k}$ are two CDFs such that for all $k$, $F^{0}_{k}\leq \tilde{F}^{0}_{k}$ and 
\begin{equation}\label{e.weakmon}
0\leq F^{0}_{k}-F^{0}_{k-1}\leq p^{*}\quad\text{and}\quad 0\leq \tilde{F}^{0}_{k}-\tilde{F}^{0}_{k-1}\leq p^{*}\ , 
\end{equation}
then Lemma \ref{lem:StochasticOrdering} guarantees that $F^{1}_{k}\leq \tilde{F}^{1}_{k}$, and Lemma \ref{lem:discrete1stepdynamics} guarantees that $F^{1}_{k}$ and $\tilde{F}^{1}_{k}$ both satisfy \eqref{e.weakmon}. We may then conclude (by induction) that $F^n_k\le \tilde{F}^n_k$ for all $n \in \N$ and $k \in \Z$, so 
the ACM evolution is stochastically monotone for the  corresponding initial distributions $\mu$ and $\tilde{\mu}$.
\end{remark}

Now we are ready to prove Proposition \ref{prop:goodSingIC_conv}. We will do so by relating the process starting from a $p^*$-bounded initial condition to a sequence of processes which begin from a discretization of a Lipschitz function. We will then be able to use Proposition \ref{p.intlip} for Lipschitz continuous initial data to yield convergence in this extended setting.

\begin{proof}[Proof of Proposition \ref{prop:goodSingIC_conv}]
Fix $\varepsilon > 0$. 
Let $(X_n,n \ge 0)$ be \acm-distributed. Then the collection of values $F^n_k=\P(X_n < k)$ satisfy the recursive relationship \eqref{eq:Scheme} with initial condition $F^0_k=\P(X_0<k)=\mu(-\infty,k)$. 
We note that the values $F^n_k$ also satisfy \eqref{eq:Scheme3} when $\xmesh$ and $\tmesh$ are chosen so that $(\xmesh)^{m+1}=\tmesh$. We enforce this relation between $\xmesh$ and $\tmesh$ throughout the proof. 

We will sandwich $F^n_k$ between two solutions of \eqref{eq:Scheme3} with smoother initial conditions. 
To this end, define $u^{\varepsilon,1}$ as the lsc viscosity solution of 
\begin{equation*}
\begin{cases}
u^{\varepsilon,1}_t + \cons\left|u^{\varepsilon,1}_x\right|^{m+1} = 0&\text{in $\RR\times (0, \infty)$,}\\
u^{\varepsilon,1}(x,0) = \varepsilon\1_{\left\{x\leq 0\right\}} + \1_{\left\{x>0\right\}}&\text{in $\RR$,}
\end{cases}
\end{equation*}
and let $u^{0,1-\varepsilon}$ denote the lsc viscosity solution of 
\begin{equation*}
\begin{cases}
u^{0,1-\varepsilon}_t + \cons\left|u^{0,1-\varepsilon}_x\right|^{m+1} = 0&\text{in $\RR\times (0, \infty)$,}\\
u^{0,1-\varepsilon}(x,0) = (1-\varepsilon)\1_{\left\{x>0\right\}}&\text{in $\RR$.}
\end{cases}
\end{equation*}
Setting 
\begin{equation}\label{eq:Sdefn}
S := S(\varepsilon) = (1-\varepsilon)^{m/(m+1)}(m+1)\cons^{1/(m+1)}m^{-m/(m+1)}\varepsilon^{1/(m+1)}, 
\end{equation}
by \eqref{eq:scaledICsoln}, these solutions have the explicit forms 
\begin{equation}\label{eq:scaledICsoln1}
u^{\ve,1}(x,t)=\begin{cases}
\ve&\text{if $x\leq 0$}, \\
\ve+\frac{m}{\cons^{\frac{1}{m}}(m+1)^{\frac{m+1}{m}}}\left(\frac{x^{m+1}}{t}\right)^{\frac{1}{m}}&\text{if $0\leq x \leq S(\varepsilon)(t/\varepsilon)^{1/(m+1)}$,}\\
1&\text{otherwise,}
\end{cases}
\end{equation}
and 
\begin{equation}\label{eq:scaledICsoln2}
u^{0,1-\ve}(x,t)=\begin{cases}
0&\text{if $x\leq 0$}, \\
\frac{m}{\cons^{\frac{1}{m}}(m+1)^{\frac{m+1}{m}}}\left(\frac{x^{m+1}}{t}\right)^{\frac{1}{m}}&\text{if $0\leq x \leq S(\varepsilon)(t/\varepsilon)^{1/(m+1)}$,}\\
1-\ve&\text{otherwise.}
\end{cases}
\end{equation}
In particular, $u^{\varepsilon,1}(x,t)=u^{0,1-\varepsilon}(x,t)+\varepsilon$. We also see from 
\eqref{eq:scaledICsoln1} and \eqref{eq:scaledICsoln2} 
that both $u^{\varepsilon,1}(x,\varepsilon), u^{0,1-\varepsilon}(x,\varepsilon)$ are Lipschitz continuous with the same Lipschitz constant $K=K(\ve)$, and therefore, there exists an $\eta=\eta(\ve)$ sufficiently small such that if $\xmesh \leq \eta$, then 
\begin{equation}\label{e.discretelip}
\begin{cases}
0\leq u^{\varepsilon,1}(x+\xmesh,\varepsilon) - u^{\varepsilon,1}(x,\varepsilon) \leq p^*&\text{for all $x\in \RR$,}\\ 
0\leq u^{0,1-\varepsilon}(x+\xmesh,\varepsilon) - u^{0,1-\varepsilon}(x,\varepsilon) \leq p^*&\text{for all $x\in \RR$.}
\end{cases}
\end{equation} 
Also, by our explicit representation of $u^{\varepsilon,1}$ in \eqref{eq:scaledICsoln1}, we have 
\begin{equation}\label{e.uve1bnd}
\begin{aligned}
&u^{\varepsilon,1}(x,\varepsilon) \geq \varepsilon\quad \text{for all $x$,}\\
&u^{\varepsilon,1}(x,\varepsilon) =1\quad\text{for all $x \geq S(\varepsilon)$}. 
\end{aligned}
\end{equation}
Now, define 
\begin{align*}
L & := L(\varepsilon) = \max\{k\leq 0:F^0_k\leq \varepsilon\}, \\
R & := R(\varepsilon) = \min\{k\geq 0: F^0_k \geq 1-
\varepsilon\}.
\end{align*}
These values are both finite because $\mu$ is a probability distribution on $\Z$ and hence $\lim_{k \to -\infty} F^0_k=0$ and $\lim_{k \to \infty} F^0_k=1$.
Then, for $\tilde n \in \N$, let $F^{+, n}_k(\tilde n)$  and $F^{-, n}_k(\tilde n)$ be the schemes defined by \eqref{eq:Scheme3}, with $(\xmesh)^{m+1}=\tmesh$ as always, and with initial conditions 
\begin{equation}\label{eq:U+IC}
F^{+, 0}_k(\tilde n)= u^{\varepsilon,1}(k\xmesh -L\tilde n^{-1/(m+1)}+S,\varepsilon) 
\end{equation}
and 
\begin{equation}\label{eq:U-IC}
F^{-,0}_k (\tilde n)= u^{0,1-\varepsilon}(k\xmesh-R\tilde n^{-1/(m+1)}, \ve),
\end{equation}
respectively. We use the parameter $\tilde n$ to spatially shift the initial conditions in order to obtain ordered initial conditions. 

By the definition of $L$, for $k < L$ we have 
\[
F^0_k \le \varepsilon \le F^{+, 0}_k(\tilde n), 
\]
the second inequality holds because $F^{+, 0}_k(\tilde n) \ge \varepsilon$ for all $k$ by \eqref{e.uve1bnd}. For $k \ge L$, if $\xmesh \le \tilde{n}^{-1/(m+1)}$ then since $L \le 0$ we have 
$k \xmesh - L\tilde{n}^{-1/(m+1)}+S \ge S$, so 
 also by \eqref{e.uve1bnd},
\[
F^{+, 0}_k(\tilde n)= u^{\varepsilon,1}(k\xmesh -L\tilde n^{-1/(m+1)}+S,\varepsilon) = 1 \ge F^0_k\, .
\]
Therefore, if $\xmesh \le \tilde{n}^{-1/(m+1)}$ then $F^{+, 0}_k(\tilde n) \ge F^0_k$ for all $k \in \Z$. 

Similarly, by the definition of $R$ and \eqref{eq:scaledICsoln2}, for $k > R$, we have 
\[
F^{-, 0}_k(\tilde n) \le 1-\varepsilon \le F^0_k\, ,
\]
and if $\xmesh \le \tilde{n}^{-1/(m+1)}$, then for $k \le R$ we have $k\xmesh-R\tilde{n}^{-1/(m+1)}\le 0$, so by \eqref{eq:scaledICsoln2},
\[
F^{-, 0}_k(\tilde n) = 0 \le F^0_k.
\]
Thus if $\xmesh \le \tilde{n}^{-1/(m+1)}$ then 
$F^{-, 0}_k(\tilde n) \le F^0_k$ for all $k \in \Z$.

Combining the two preceding paragraphs, we obtain that if $\xmesh \le \tilde{n}^{-1/(m+1)}$ then for all $k$, 
\begin{equation*}\label{e.sandwich_new}
F^{-,0}_k(\tilde n) \leq F^{0}_k \leq F^{+,0}_k(\tilde n)\, .
\end{equation*}
If also $\xmesh \le \eta$, then 
by \eqref{e.discretelip}, each scheme satisfies the condition that 
\[
|F^{\pm,0}_k(\tilde{n})-F^{\pm,0}_{k-1}(\tilde{n})|\leq p^*
\] 
for all $k \in \Z$. By Remark \ref{r.newmon}, the prior two displays yield that
whenever $\xmesh \leq \min(\eta, 1/\tilde n^{(m+1)})$, we have by induction that for all $n\in \N$, 
\begin{equation}\label{eq:schemeordering_new}
F^{-,n}_k(\tilde n) \leq F^{n}_k \leq F^{+,n}_k(\tilde n).
\end{equation}

We now combine these bounds with Proposition \ref{p.intlip}. We first aim to apply the proposition with $\mu_N$ defined by 
\[
\mu_N[-\infty,k)= F^{+,0}_k(\tilde{n})\, .
\]
The proposition requires that $\mu_N$ have the form $\mu_N[-\infty,k) = u_0(k/N^{1/(m+1)})$, so the definition of $F^{+,0}_k(\tilde{n})$ forces us to take $\xmesh=N^{-1/(m+1)}$ and $u_0(x) = u^{\eps,1}(x-L/\tilde{n}^{1/(m+1)}+S,\eps)$. Since $u^{\eps,1}$ is Lipschitz, fixing $T > 1$ and applying Proposition \ref{p.intlip} (specifically \eqref{e.invp}) at time $t = 1 \in [0,T]$, it follows that there exist $N_0=N_0(q,m,K)$ and $c=c(K, m, T)$  such that if $N \ge N_0$, for all $k \in \Z$, 
\[
F^{+,N}_k(\tilde{n}) \le  u^{\eps,1}(kN^{-1/(m+1)} - L\tilde n^{-1/(m+1)}+S, 1 + \eps) +  cN^{-1/2}. 
\]
We emphasize that $N_0$ and $c$ depend only on the initial condition $u_0(x) = u^{\eps,1}(x-L/\tilde{n}^{1/(m+1)}+S,\eps)$ through its Lipschitz constant $K$; in particular, $N_0$ and $c$ do not depend on $\tilde{n}$ since varying $\tilde{n}$ translates the  initial condition horizontally but does not change its Lipschitz constant. 

Similarly, taking $\xmesh=N^{-1/(m+1)}$ and $u_0(x) = u^{0,1-\eps}(x-R/\tilde{n}^{1/(m+1)},\eps)$ and $t=1$, applying Proposition \ref{p.intlip} (specifically \eqref{e.invp}) with $\mu_N$ defined by $\mu_N[-\infty,k)= F^{-,0}_k(\tilde{n})$ yields that for all $N \ge N_0$ and all $k \in \Z$, 
\[
F^{-,N}_k(\tilde{n}) \ge  u^{0,1-\eps}(kN^{-1/(m+1)} - R\tilde n^{-1/(m+1)}, 1 + \eps) -  cN^{-1/2}. 
\]
For $N\ge N_0$ large enough that also $\xmesh = N^{-1/(m+1)} \le \min(\eta, \tilde n^{-1/(m+1)})$, we may combine these bounds with \eqref{eq:schemeordering_new} to deduce that for all $k \in \Z$, 
\begin{align*}
\P(X_N < k) = F^N_k \leq F^{+,N}_k(\tilde n) & \leq u^{\varepsilon,1}(kN^{-1/(m+1)} - L\tilde n^{-1/(m+1)}+S, 1 + \varepsilon) +  cN^{-1/2} 
\end{align*}
and 
\begin{align*}
\P(X_N < k) = F^N_k \geq F^{-,N}_k(\tilde n) & \geq u^{0,1-\varepsilon}(kN^{-1/(m+1)} - R\tilde{n}^{-1/(m+1)},1+\varepsilon) -  cN^{-1/2}. 
\end{align*}

Taking $k = xN^{1/(m+1)}$, these bounds become 
\[
\P(X_N < xN^{1/(m+1)}) \leq u^{\varepsilon,1}(x - L\tilde n^{-1/(m+1)} + S,1+\varepsilon)+ cN^{-1/2},
\]
and
\[
\P(X_N < xN^{1/(m+1)}) \geq u^{0,1-\varepsilon}(x - R \tilde n^{-1/(m+1)},1+\varepsilon)- cN^{-1/2}.
\]
We should in fact take $k = \lfloor xN^{1/(m+1)}\rfloor$ above, but we ignore this minor rounding issue to preserve readability, as the errors it creates are asymptotically negligible for $N$ large due to the spatial continuity of $u^{\eps,1}$ and of $u^{0,1-\eps}$ at time $1+\eps$.

For $\tilde{n} \ge \max(N_0,\eta^{-(m+1)})$, if $N \ge \tilde{n}$ then the other constraints on $N$ are automatically satisfied.
Recalling that $L \le 0$, since $cN^{-1/2} \to 0$ as $N \to \infty$, the first of the two preceding bounds then implies that 
\begin{align*}
\limsup_{N\to\infty} \P\left(\frac{X_N}{N^{1/(m+1)}} < x\right) 
& 
\le \inf_{\tilde{n} \ge \max(N_0,\eta^{-(m+1)})}
\Big(
u^{\varepsilon,1}(x - L\tilde n^{-1/(m+1)} + S,1+\varepsilon) \Big)\\
& =u^{\varepsilon,1}(x+S,1+\varepsilon). 
\end{align*}
Likewise, the second of the bounds yields that 
\begin{align*}
\liminf_{N\to\infty} \P\left(\frac{X_N}{N^{1/(m+1)}} < x\right) 
& \ge 
\sup_{\tilde{n} \ge \max(N_0,\eta^{-(m+1)})}
\Big(
u^{0,1-\varepsilon}(x - R \tilde n^{-1/(m+1)},1+\varepsilon) 
\Big)
\\
& 
= u^{0,1-\varepsilon}(x,1+\varepsilon)\, .
\end{align*}

Finally, from the explicit representations of $u^{\varepsilon,1}, u^{0,1-\varepsilon}$ from \eqref{eq:scaledICsoln1} and \eqref{eq:scaledICsoln2}, and $S=S(\ve)$ defined by \eqref{eq:Sdefn}, we have
\begin{align*}
\lim_{\varepsilon\to 0} S(\varepsilon)= 0\\
\lim_{\varepsilon\to 0} u^{\varepsilon,1}(x+S(\varepsilon),1+\varepsilon) = u(x,1)\\
\lim_{\varepsilon\to 0} u^{0,1-\varepsilon}(x,1+\varepsilon) = u(x,1),
\end{align*}
uniformly in $x$, where $u(x,t)$ is given by \eqref{eq:explicitsolution}. Taking the limit as $\varepsilon \to 0$, we get that 
\[
u(x,1) \leq \liminf_{n\to\infty} \P\left(\frac{X_n}{n^{1/(m+1)}} < x\right) \leq \limsup_{n\to\infty} \P\left(\frac{X_n}{n^{1/(m+1)}} < x\right) \leq u(x,1),
\]
and therefore 
\[
\lim_{n\to\infty} \P\left(\frac{X_n}{n^{1/(m+1)}} < x\right) = u(x,1)\, ,
\]
as desired. 
\end{proof}

\section{General singular initial conditions}\label{sect:GenIC}

We saw that Lemma \ref{lem:StochasticOrdering} requires a bound on the maximum single-site probability. 
Our next result shows that in fact, there exists a constant $N$ such that, regardless of the initial distribution, the distribution of the ACM after $N$ steps will satisfy such a bound. 
\begin{lemma}\label{lem:discreteICevolution}
Let $\mu$ be a probability distribution supported on $\Z$, and 
let $(X_n, n\geq 0)$ be \acm-distributed.   Then there exists a constant $N=N(m,q,r)$, such that for all $n \geq N$,
\begin{equation} \label{eq:bddnesscondition}  
\max_{k\in\mathbb{Z}} \P(X_n = k) \leq p^*.
\end{equation} 
\end{lemma}

In the proof of Lemma \ref{lem:discreteICevolution}, we want to use the fact that $p^* > 1/2$ to bound the decrease in the maximum away from 0. However, we must also consider the case when $p^* = 1/2$, which occurs when $m = q = 1$. As such, the $\mathrm{ACM}(1,1,r,\mu)$-distributed processes will be addressed in a slightly different way.

\begin{proof}[Proof of Lemma \ref{lem:discreteICevolution}]
We know by Lemma \ref{lem:discrete1stepdynamics} that if $M:= \max_{k\in\Z} p_k^0 \leq p^*$, then $\max_{k \in \Z} p^n_k \leq p^*$ for all $n$.  So suppose there exists an $\ell \in \Z$ such that $M = P(X_0 = \ell) > p^*$. We will show that 
\begin{equation}
\max_{k \in Z} \P(X_1 = k) \leq \max(p^*,M-C)
\end{equation}
for some $C = C(m,q,r) > 0$. The full result will then follow by induction, as this implies that for all $n \in \N$, 
\[
\max_{k \in \Z} p^n_k \le \max(p^*,M-nC), 
\]
and in particular $\max_{k \in \Z} p^n_k\leq p^{*}$ for all $n \ge (1-p^*)/C$.

We know that because $p^*\geq 1/2$, there exists a unique $\ell$ such that $p^0_\ell > p^*$. 

First suppose that $p^* > 1/2$. We consider each site $k \in \Z$ and show that either $p^1_k \leq p^*$ or $M - p^1_k = p^0_\ell - p^1_k \geq C$. 
 
\begin{itemize}
\item When $k \not \in \{\ell-1,\ell,\ell+1\}$, then by Lemma \ref{lem:discrete1stepdynamics} we have $$p^1_k \leq \max(p^0_{k-1},p^0_k,p^0_{k+1}) \leq p^*.$$ 

\item For the value $p^1_\ell$, we have 
 
\begin{align*}
p^0_\ell - p^{1}_\ell &= -q\left[r(p^0_{\ell-1})^{m+1}-(p^0_\ell)^{m+1}+(1-r)(p^0_{\ell+1})^{m+1}\right]\\
&\geq -q\left[r(1/2)^{m+1}-(p^*)^{m+1}+(1-r)(1/2)^{m+1}\right]\\
&= q\left[(p^*)^{m+1}-(1/2)^{m+1}\right]\\
&:=C_1> 0.
\end{align*}

\item For site $\ell+1$, we have a similar computation:
 
\begin{align*}
p^0_{\ell} - p^{1}_{\ell+1} &= p^0_{\ell}-\left[p^0_{\ell+1}+ qr(p^0_{\ell})^{m+1}-q(p^0_{\ell+1})^{m+1}+(1-r)q(p^0_{\ell+2})^{m+1}\right]\\
&= p^0_{\ell}-\left[p^0_{\ell+1}+ qr(p^0_{\ell})^{m+1}-q(p^0_{\ell+1})^{m+1} +(1-r)q(p^0_{\ell+2})^{m+1}\right]\\
&-(1-r)q(p^0_{\ell})^{m+1} + (1-r)q(p^0_{\ell})^{m+1}\\
&= \left[p^0_{\ell}-q(p^0_{\ell})^{m+1}\right] - \left[p^0_{\ell+1}-q(p^0_{\ell+1})^{m+1}\right] + (1-r)q\left[(p^0_{\ell})^{m+1} - (p^0_{\ell+2})^{m+1}\right]\\
&= \gfunc(p^0_\ell,p^0_{\ell+1})+ (1-r)q\left[(p^0_{\ell})^{m+1} - (p^0_{\ell+2})^{m+1}\right]\\
&\geq \min(\gfunc(p^*,1-p^*),1-q) + (1-r)q\left[(p^*)^{m+1} - (1/2)^{m+1}\right]\\
&:= C_2,
\end{align*}
where the inequality is obtained by Lemma \ref{lem:h(x,y)facts}. Now notice that each term in $C_2$ is non-negative. If $q < 1$, the first term is strictly positive and if $q = 1$, then $r \neq 1$, so the second term is strictly positive. Therefore $C_2 > 0$. 

\item We also bound $p^0_\ell - p^1_{\ell-1}$ with a similar computation. 

\begin{align*}
p^0_{\ell} - p^{1}_{\ell-1} &= p^0_{\ell}-\left[p^0_{\ell-1}+ qr(p^0_{\ell-2})^{m+1}-q(p^0_{\ell-1})^{m+1}+(1-r)q(p^0_{\ell})^{m+1}\right]\\
&= p^0_{\ell}-\left[p^0_{\ell-1}+ qr(p^0_{\ell-2})^{m+1}-q(p^0_{\ell-1})^{m+1} +(1-r)q(p^0_{\ell})^{m+1}\right]\\
&-qr(p^0_{\ell})^{m+1} + qr(p^0_{\ell})^{m+1}\\
&= \left[p^0_{\ell}-q(p^0_{\ell})^{m+1}\right] - \left[p^0_{\ell-1}-q(p^0_{\ell-1})^{m+1}\right]+ qr\left[(p^0_{\ell})^{m+1} - (p^0_{\ell-2})^{m+1}\right]\\
&\geq \min(\gfunc(p^*,1-p^*),1-q) + qr\left[(p^*)^{m+1} - (1/2)^{m+1}\right]\\
&:= C_3.
\end{align*}

Again, because $\min(\gfunc(p^*,1-p^*),1-q) \geq 0$ and the second term in $C_3$ is strictly positive, we can see that $C_3 > 0$.
\end{itemize}
Therefore, if we let $C := \min(C_1, C_2, C_3)>0$, we have shown that if $p^* > 1/2$

$$\max_{k \in Z} \P(X_1 = k) \leq \max(p^*,M-C)$$

as desired.

We now address the case when $p^* = 1/2$; so $m = q = 1$. Again, Lemma \ref{lem:discrete1stepdynamics} tells us that for $k \not \in \{\ell-1,\ell,\ell+1\}$, $p^1_k \leq p^*$. For $k \in \{\ell-1,\ell,\ell+1\}$ we again bound $p^0_\ell - p^1_k$ from below, but the bounds used above are too loose. So we use the method of Lagrange multipliers to get a better bound. The details are routine but tedious so we only sketch them. In three separate calculations, we minimize the function of interest subject to the constraints that the sum of the three sites must be less than 1, that $1/2 \leq p^0_\ell \leq 1$, and that all sites have a non-negative mass. The functions we choose to minimize will be
\begin{align*}
p^0_\ell - p^1_{\ell} &= - r(p^0_{\ell-1})^2 +(p^0_{\ell})^2 -(1-r)(p^0_{\ell+1})^2,\\
p^0_\ell - p^1_{\ell-1} &= p^0_\ell - p^0_{\ell-1} - r(p^0_{\ell-2})^2 +(p^0_{\ell-1})^2 -(1-r)(p^0_\ell)^2,\\
p^0_\ell - p^1_{\ell+1} &= p^0_\ell - p^0_{\ell+1} - r(p^0_{\ell})^2 +(p^0_{\ell+1})^2 -(1-r)(p^0_{\ell+2})^2.
\end{align*} 
Leaving the details to the reader, we see that each of these functions is minimized at a corner point of the corresponding constraint region. This yields the following lower bound : 

$$\min( p^0_\ell - p^1_{\ell}, p^0_\ell - p^1_{\ell-1}, p^0_\ell - p^1_{\ell+1}) \geq \min\left(\frac{1-r}{4}, \frac{r}{4}, \frac{1-r}{4}\right)=\frac{1-r}{4}.$$

Because $q = 1$ in this case, we know $r \neq 1$ and therefore in this case $C := 1/4- r/4 > 0$. With this constant, the result again follows by induction. 
\end{proof}

We now have everything needed to prove the main theorem. 

\begin{proof}[Proof of Theorem \ref{thm:main2}]
Fix $N=N(q,m,r)$ as in Lemma \ref{lem:discreteICevolution},
let $\tilde{\mu}$ be the distribution of $X_{N}$ and let $\tilde{X}_n = X_{N+n}$ for $n \ge 0$. Then $(\tilde{X}_n,n \ge 0)$ is ACM$(m,q,r,\tilde{\mu})$-distributed.
Because $\P(\tilde{X}_0=k) \leq p^*$ for all $k$, we can apply Proposition \ref{prop:goodSingIC_conv} to conclude that 
$$\lim_{n\to\infty} \P\left(\frac{\tilde{X}_n}{n^{1/(m+1)}} < x \right) = u(x,1).$$
Since $N$ is fixed and $u(x,1)$ is continuous, this implies that for all $x \in \R$, 
\begin{equation}\label{eq:final_unscaled_conv}
\lim_{n\to\infty} \P\left(\frac{X_n}{n^{1/(m+1)}} < x \right) 
= 
\lim_{n \to \infty}
\P\left(\frac{\tilde{X}_n}{n^{1/(m+1)}} < x \Big(\frac{n-N}{n}\Big)^{1/(m+1)}\right) 
= u(x,1).
\end{equation}

By comparing the expression for $u(x,1)$ provided by \eqref{eq:explicitsolution} to the CDF given in \eqref{e.cdfbeta} for $(m+1)\left(\frac{\cons}{m^m}\right)^{1/m+1}B$,  where $B$ is $\Bet(\tfrac{m+1}{m},1)$-distributed, we see that 
\[
 \frac{1}{m+1}\left(\frac{m^m}{\cons}\right)^{\frac{1}{m+1}} \cdot
\frac{X_n}{n^{1/(m+1)}}\xrightarrow{d} B,
\]
as required. 
\end{proof}

\begin{remark}\label{r.smallr}
When $0 \leq r < 1/2$ we have $\cons < 0$, and the PDE arguments above don't apply directly, as $H(p) = \cons |p|^{m+1}$ is no longer convex. However, we can prove the result for $r$ in this range as follows. 

Let $(X_n, n \geq 0)$ be \acm-distributed with $r < 1/2$. Define $(Y_n, n \geq 0)$ with $Y_n = -X_n$ for all $n$. Then $(Y_n, n\geq 0)$ is $\text{ACM}(m,q,1-r,-\mu)$ distributed, where we take the convention that $-\mu(x):=\mu(-x)$. Because $1 - r > 1/2$, the arguments above show that 
$$\frac{1}{m+1}\left(\frac{m^m}{(2(1-r)-1)q}\right)^{\frac{1}{m+1}}\cdot \frac{Y_n}{n^{1/(m+1)}} \xrightarrow{d} B$$
where $B$ is $\Bet\left(\frac{m+1}{m},1\right)$-distributed. 

Noting that $Y_n = -X_n$, we simplify and plug in to see that this is equivalent to the statement

$$\frac{1}{m+1}\left(\frac{m^m}{-(2r-1)q}\right)^{\frac{1}{m+1}}\cdot \frac{-X_n}{n^{1/(m+1)}} \xrightarrow{d} B.$$

Now, recalling that $\cons = (2r-1)q < 0$ when $r < 1/2$, we see that this can be rewritten as

$$\frac{\text{sign}(\cons)}{m+1}\left(\frac{m^m}{|\cons|}\right)^{\frac{1}{m+1}}\cdot \frac{X_n}{n^{1/(m+1)}} \xrightarrow{d} B.$$

This is precisely the result in Theorem \ref{thm:main2}.

\end{remark}

\section{Generalizations, Limitations, and Open Questions}\label{s.generalization}

In this section, we discuss several possible extensions to the above results, as well as some obstacles and challenges we have observed.

\noindent \emph{1. Higher Dimensions.} One may try to extend our results and techniques to higher space-dimensions. However, there are several challenges. In particular, we use the monotonicity of CDFs, namely that $F^{n}_{k}$ is nondecreasing in $k$, ubiquitously throughout the paper. The monotonicity of CDFs in higher dimensions is weaker, as it requires ordering in all coordinates.
Moreover, the natural recurrence for the CDF in higher dimensions requires information about the values of the CDF along the boundary of a quadrant (since such a CDF encodes the probability that a random variable lies in a quadrant). Thus, such a recurrence is highly non-local and, as such, does not seem amenable to PDE tools we use in this paper. 
Relatedly, it is unclear how to extend the stochastic monotonicity and sandwiching arguments from this paper to higher dimensions. 
 These points make the extension of our results to dimensions $d > 1$ delicate (although we hope not impossible).

 \noindent \emph{2. Connection to zero-range processes.} The following interacting particle system is, heuristically, a relative of cooperative motion. Begin with $L$ particles at the origin $0 \in \Z$. For each time $n \in \N$, choose $m+1$ particles independently and uniformly at random. If all $m+1$ particles are at the same location, then one of the $m+1$ particles makes a biased lazy random walk step ($1$ with probability $rq$, $-1$ with probability $(1-r)q$, otherwise $0$). If the $m+1$ particles are spread among at least two sites then no motion occurs. Up to a time-change, this process is equivalent to a zero range process on $\Z$, where the rate of motion off a site with $k$ particles is $r(k)=k^{m+1}$.
 
Write $P^n_k$ for the proportion of the $L$ particles on site $k$ at time $n$. Then a simple one-step calculation using the tower law gives that 
 \[
\E(P^{n+1}_k-P^n_k) = 
\frac{1}{L}\E\left(-rq\left((P^n_k)^{m+1}-(P^n_{k-1})^{m+1}\right) + (1-r)q\left((P^n_{k+1})^{m+1}-(P^n_{k})^{m+1}\right)\right). 
\]
If the random variables $P^n_k$ are well-concentrated around their expectations, then this suggests that, after a time-change, the empirical particle density $P^n_k$ should evolve like the solution of \eqref{eq:pmf_formulation}. Results of this form --- hydrodynamic limits for the empirical particle density of zero-range processes --- are known in some cases (\cite{MR1707314} presents results of this kind, and \cite{MR1081251} analyzes a related, $\R^d$-valued model), but we were unable to find any existing results which apply to dynamics which are as singular as those described above.

\noindent \emph{3. Cooperative Motion with fewer than $m$ friends.} A related model which we have not considered, but which may be amenable to the techniques of this paper, is when the cooperative motion only requires $\ell$ individuals to move, for $\ell<m$. More precisely, we may modify the model as follows. Let $X_0$ and $(D_n,n \ge 0)$ be as in the introduction. Then, for $n \ge 0$, let  $(\tilde{X}^i_n,1 \le i \le m)$ be independent copies of $X_n$, and set 
\begin{equation*}
X_{n+1} = \begin{cases} X_n + D_n &\text{if } X_n = \tilde{X}^i_n \text{ for at least $\ell$ distinct values } i\in\{ 1 \dots m\},\\
X_n & \text{otherwise} \, .
\end{cases}
\end{equation*}
It seems likely that for such a process, $X_n$ should typically take values of order $n^{1/(\ell+1)}$. A heuristic argument for this is as follows. Suppose that $X_n/n^{\alpha}$ behaves roughly like a continuous random variable with compact support, for $n$ large; say that $\P(X_n=k)\asymp n^{-1/\alpha}$ for $\Theta(n^{\alpha})$ distinct values of $k$, and for other values of $k$ this probability is substantially smaller. 

On one hand, this suggests that $\P(X_{n+1} > X_n) = \Theta(n^{\alpha-1})$, since we expect that $X_{2n}-X_n = \Theta(n^{\alpha})$. On the other hand, $\P(X_{n+1} > X_n)$ is the probability that at least $\ell$ of the $m$ copies of $\tilde{X}_n^i$ take the same value as $X_n$; if the distribution of $X_n$ is spread out over roughly $n^{\alpha}$ sites, then this probability should be around $(n^{-\alpha})^{\ell}$. For these two predictions to agree we must have $1-\alpha = \alpha\ell$, so $\alpha=1/(\ell+1)$. 

While we have confidence in this prediction of the asymptotic size of $X_n$, it is not clear to us whether or not the scaling limit should in fact be the same as for a $\mathrm{ACM}(\ell,q,r,\mu)$ process.

\noindent \emph{4. Cooperation Motion with a non-integer number of friends}.
Another possible extension is to the case when $m$ is non-integer. Although $m$ integer has a natural interpretation in terms of cooperative motion, which in turn leads to the recursion relation \eqref{e.fnk1st}, we may alternatively take \eqref{e.fnk1st} as a definition for the CDF $F^{n}_{k}$ of a random variable $X_{n}$. In this case, the same analysis shows that for any $m \in \R$ with $m \ge 1$, Theorem \ref{thm:main2} still holds when $X_{0}$ is  $\mu$-distributed and $\P(X_{n}<k)=F^{n}_{k}$, where $F^n_k$ is defined according to \eqref{e.fnk1st}. 

The requirement that $m\geq 1$ is crucial for Lemma \ref{lem:h(x,y)facts}, so the techniques of this paper do not yield insight into what happens for $m \in (0,1)$.  However, it would be quite interesting to understand this, as well as the limiting behaviour as $m\rightarrow 0$.

\noindent \emph{5. General step size distributions.} 
Remaining in one spatial dimension, another natural generalization of this process would be to consider cooperative motion which allowed for more general step sizes $D_n$. As we will discuss in Section \ref{ss.big1}, it turns out that if $\P(|D_{n}|>1)>0$, we confront an immediate, provable obstacle to directly applying the proof techniques of this paper (failure of monotonicity). However, before describing this obstacle, we first present a generalization of our main result. If the steps $(D_n,n \ge 0)$ take on the values $+R,0,-R$ for some integer $R$, instead of $+1,0,-1$, then we are able to prove a distributional convergence result; this is presented in the next subsection. Surprisingly, the limiting distribution in this case, although always a mixture of Beta random variables, need not be Beta-distributed, due to lattice effects which persist at large times. 
 
 \subsection{Persistent lattice effects}\label{sec:lattice}

Let $(D_n,n \ge 0)$ be iid bounded integer random variables. Define a cooperative motion process, with $X_{0}$ chosen according to an initial probability distribution $\mu$ on $\Z$, as follows. For $n \ge 0$, let  $(\tilde{X}^i_n,1 \le i \le m)$ be independent copies of $X_n$, and set 
\begin{equation}
X_{n+1} = \begin{cases} X_n + D_n &\text{if } X_n = \tilde{X}^i_n \text{ for all } i= 1 \dots m,\\
X_n &\text{if } X_n \neq \tilde{X}^i_n \text{ for some } i \, .
\end{cases}
\label{eq:GeneralRecursionDefn}
\end{equation}

In this section we consider the case where the steps $(D_n,n \ge 0)$ are iid with $\P(D_n=\constwo)=rq$, $\P(D_n=0) = 1-q$, and $\P(D_n = -\constwo)=(1-r)q$ for some $q \in (0,1]$, $r \in (1/2,1]$ with $r$ and $q$ not both 1, and $\constwo \in \N$. If the initial distribution $\mu$ is supported by a translate of $\constwo\Z$ then the resulting cooperative motion may simply be seen as a rescaling of the Bernoulli cooperative motion process considered in the body of the paper. However, if $\mu$ is not supported by a translate of $\constwo\Z$ then the asymptotic behaviour is in fact different; there are lattice effects which persist at large times. 
\begin{thm}\label{thm:main_lattice}
Consider the generalized cooperative motion with 
$\P(D_n=\constwo)=rq$, $\P(D_n=0) = 1-q$, and $\P(D_n = -\constwo)=(1-r)q$ for some $q \in (0,1]$, $r \in (1/2,1]$ with $r$ and $q$ not both 1, and $\constwo \in \N$. Write $\pi_r = \P(X_0=r\mod \constwo)$ for $r \in \{1,2,\ldots,\constwo\}$. 
Then 
\[
\frac{1}{\constwo} \frac{1}{m+1} \left(\frac{m^m}{\cons n}\right)^{1/(m+1)} X_n 
\xrightarrow{d}
B \cdot \sum_{r=1}^\constwo (\pi_r)^{m/(m+1)}\indc_{\left\{A=r\right\}},
\]
where $A$ is a random variable taking values in $\{1,2,\ldots,\constwo\}$ with $\P(A=r)=\pi_r$, and $B$ is $\Bet(\tfrac{m+1}{m},1)$-distributed and independent of $A$. 
\end{thm}
As an input to the proof of Theorem~\ref{thm:main_lattice}, we use the following straightforward extension of Theorem~\ref{thm:main2} to Bernoulli cooperative motions which may take values $\pm \infty$. 
Let $c=c(q,m,r) = m\cons^{-1/m}(m+1)^{-(m+1)/m}$, and for $0 \le a < b \le 1$ define an extended CDF $F^{a,b}$ by 
\[
F^{a,b}(x) = 
\begin{cases}
a					& \mbox{if }x \le 0\\
a + cx^{(m+1)/m}	& \mbox{if }0 \le cx^{(m+1)/m} \le b-a \\
b					& \mbox{if }b-a \le cx^{(m+1)/m}\, .
\end{cases}
\]
Let $B^{a,b}$ be an extended random variable with distribution $F^{a,b}$. Then $\P(|B^{a,b}| < \infty) = b-a$, and for $x \in \R$,
\[
\P(B^{a,b} \le x~|~|B^{a,b}|<\infty)
= 
\begin{cases}
0					& \mbox{if }x \le 0\\
\frac{c}{b-a}x^{(m+1)/m}	& \mbox{if }0 \le \frac{c}{b-a}x^{(m+1)/m} \le 1 \\
1					& \mbox{if }1 \le \frac{c}{b-a}x^{(m+1)/m}\, .
\end{cases}
\]
In other words, given that $|B^{a,b}|$ is finite, it is distributed as $(\tfrac{b-a}{c})^{m/(m+1)}B$ where $B$ is $\Bet(\tfrac{m+1}{m},1)$-distributed
\begin{prop}\label{thm:main_extended}
If $\mu$ is a probability distribution on $\Z\cup\{\pm \infty\}$ with $\mu(\{-\infty\})=a$ and $\mu(\{+\infty\})=1-b$, and $(X_n,n \ge 0)$ is $\acm$-distributed, then 
\[
 \frac{X_n}{n^{1/(m+1)}} 
\xrightarrow{d} B^{a,b}. 
\]
\end{prop}
The proof of Proposition~\ref{thm:main_extended} proceeds exactly as does the proof of Theorem~\ref{thm:main2}, with minor notational changes, so we omit the details. 
\begin{cor}\label{cor:main_extended}
Suppose that $\mu$ is a probability distribution on $\Z\cup\{\pm \infty\}$ with $\mu(\{-\infty\})=a$ and $\mu(\{+\infty\})=1-b$, and that $\P(D_n=\constwo)=rq$, $\P(D_n=0)=1-q$, and $\P(D_n=-\constwo)=(1-r)q$ for some $q \in (0,1)$, $r \in (1/2,1]$ with $r$ and $q$ not both 1, and $\constwo \in \N$, $\constwo > 0$. If there is $r \in \{1,2,\ldots,\constwo\}$ such that $\P(|X_0|=r \mod \constwo~|~|X_0|<\infty)=1$, then 
\[
\frac{1}{\constwo} \cdot \frac{X_n}{n^{1/(m+1)}} 
\xrightarrow{d} B^{a,b}. 
\]
\end{cor}

\begin{proof}
Apply Proposition~\ref{thm:main_extended} to the process $((X_n-r)/\constwo,n \ge 0)$. 
\end{proof}
\begin{proof}[Proof of Theorem~\ref{thm:main_lattice}]
Define auxiliary processes $(X_n^{(r)},n \ge 0)$ for $1 \le r \le \constwo$ by 
\[
X_n^{(r)} = 
\begin{cases}
X_n			& \mbox{ if }X_n=r\mod \constwo\\
-\infty		& \mbox{ otherwise.}
\end{cases}
\]
Then by Corollary~\ref{cor:main_extended}, for each $1 \le r \le \constwo$, 
\[
\frac{1}{\constwo} \cdot  \frac{X_n^{(r)}}{n^{1/(m+1)}}
\xrightarrow{d} B^{1-\pi_r,1}. 
\]
Moreover, since exactly one of $X_n^{(1)},\ldots,X_n^{(\constwo)}$ is finite, and $\P(|X_n^{(r)}|<\infty) = \pi^{(r)}$ for all $n \ge 0$ and $1 \le r \le \constwo$, it follows that 
\[
\left(\frac{1}{\constwo} \cdot \frac{X_n^{(r)}}{n^{1/(m+1)}},1 \le r \le \constwo\right)
\xrightarrow{d} (B^{1-\pi_r,1},1 \le r \le \constwo),
\]
where the joint distribution of the variables on the right-hand side is fully determined by the stipulation that exactly one of them is finite and all others take the value $-\infty$. 

Finally, with the convention that $(-\infty)\cdot 0=0$, we have 
\[
X_n = \sum_{r=1}^\constwo X_n^{(r)} \indc_{\left\{|X_n^{(r)}| < \infty\right\}},
\]
together with which the preceding joint convergence implies that 
\[
\frac{1}{\constwo} \frac{X_n}{n^{1/(m+1)}}
\xrightarrow{d}
\sum_{r=1}^\constwo B^{1-\pi_r,1}\indc_{\left\{|B^{1-\pi_r,1}|<\infty\right\}}\, .
\]
Since $\P(|B^{1-\pi_r,1}|<\infty)=1-(1-\pi_r)=\pi_r$ for each $r \in \{1,2,\ldots,\constwo\}$, and the conditional distribution of $B^{1-\pi_r,1}$ given that $|B^{1-\pi_r,1}|$ is finite is the same as that of $(\pi_r/c)^{m/(m+1)}\Bet(\tfrac{m+1}{m},1)$, the result follows. 
\end{proof}

\subsection{Step Sizes $|D_{n}|>1$}\label{ss.big1}
Building on our main theorem, and in view of the persistent lattice effects explained in the preceding subsection, we make the following conjecture. Consider the generalized cooperative motion defined by \eqref{eq:GeneralRecursionDefn} and write $\nu$ for the common distribution of $(D_{n}, n\geq 0)$. If $\gcd(k> 0:\P(|D_n|=k)>0)=1$, then there exists $c =c(\nu) > 0$ such that $cn^{-1/(m+1)}X_n \xrightarrow{d} B$, 
where $B$ is $\Bet\left(\frac{m+1}{m},1\right)$-distributed. 

The preceding conjecture states that all totally asymmetric cooperative motion processes with non-negative, bounded integer step sizes whose support is not contained in a proper sublattice of $\Z$ should have similar asymptotic behaviour. However, there is a provable difficulty in establishing this conjecture beyond the setting described in this paper using the proof techniques shown above.
Specifically, we next show that monotonicity of the evolution fails to hold whenever $\P(|D_n|>1) > 0$. This implies that, in some sense, the main proof technique used in this paper can only handle cooperative motion-type processes with $|D_n| \leq 1$. 

Consider a cooperative-motion type process as in \eqref{eq:GeneralRecursionDefn}, with bounded but not necessarily positive step sizes, so $\P(-\ell \le D_n \le s)=1$ for some non-negative integers $s$ and $\ell$. Writing $F^n_k = \P(X_n < k)$, then the values $F^n_k$ satisfy the following recurrence:
\begin{align}
F^{n+1}_k &= G(F^n_{k+\ell}, \dots, F^n_k, \dots, F^n_{k-s}) \notag\\ 
&:= F^n_k - \sum_{j = k-s}^{k-1} \P(X_n = j)^{m+1}\P(D_n \geq k-j) + \sum_{j = k}^{k+\ell-1} \P(X_n = j)^{m+1}\P(D_n < k-j) \notag\\
&= F^n_k - \sum_{j =k-s}^{k-1} (F^n_{j+1}-F^n_j)^{m+1}\P(D_n \geq k-j) + \sum_{j = k}^{k+\ell-1} (F^n_{j+1}-F^n_j)^{m+1}\P(D_n < k-j).\notag
\end{align}
The function $G$ is defined by the equality of the first and third lines, above: so 
\begin{align*}
& G(f_{k+\ell},\ldots,f_{k-s}) \\
& = f_k - \sum_{j = k-s}^{k-1} (f_{j+1}-f_j)^{m+1}\P(D_n \geq k-j) + \sum_{j=k}^{k+\ell-1} 
(f_{j+1}-f_j)^{m+1}\P(D_n < k-j)\, .
\end{align*}
\begin{thm}\label{thm:no_stoch_mono}
If $\P(|D_n| > 1) > 0$, then there is no $\Lambda > 0$ such that $G$ is nondecreasing in each argument whenever 
\begin{equation*}
0 \leq f_{j+1} - f_{j} \leq \Lambda
\end{equation*}
for all $j \in [k-s,k+\ell-1]$. 
\end{thm}
\begin{proof}
First, 
\begin{align*}
\frac{\partial G}{\partial f_{k-1}} &= (m+1)(f_k - f_{k-1})^m\P(D_n \geq k - (k-1))\\ 
&- (m+1)(f_{k-1}-f_{k-2})^m\P(D_n \geq k - (k-2))\\
&= (m+1)(f_k - f_{k-1})^m\P(D_n \geq 1)- (m+1)(f_{k-1}-f_{k-2})^m\P(D_n \geq 2),
\end{align*}
so if $\P(D_n \geq 2) > 0$ and if $f_k = f_{k-1} > f_{k-2}$, then 
$$\frac{\partial G}{\partial f_{k-1}} < 0.$$
Similarly,
\[
\frac{\partial G}{\partial f_{k+1}}
= 
(m+1)(f_{k+1} - f_{k})^m\P(D_n < 0)-(m+1)(f_{k+2}-f_{k+1})^m\P(D_n < -1), 
\]
so if $\P(D_n < -1) > 0$, then whenever $f_k=f_{k+1} < f_{k+2}$ then we have 
\[
\frac{\partial G}{\partial f_{k+1}} < 0.\qedhere
\] 
\end{proof}
Note that for any initial distribution with bounded support, if the step size is bounded then for all $n$ the support of $X_n$ is bounded: letting $k = \max\{\ell: F^n_{\ell} < 1\}+2$ and $k'=\min\{\ell: F^n_{\ell}>0\}-2$, then $k$ and $k'$ are both finite. Moreover, $F^n_{k-2} < F^n_{k-1} =  F^n_{k}=1$ and $0=F^n_{k'} = F^n_{k'+1} < F^n_{k'+2}$, and thus if $\P(|D_n|>1) > 0$ then by the above theorem, at no point in the evolution will the process reach a time at which monotonicity can be invoked. Without monotonicity, we can not apply the Crandall-Lions methodology, so the proof technique used in this paper fails. 

\subsection*{Acknowledgements}
LAB was partially supported by NSERC Discovery Grant 643473 and Discovery Accelerator Supplement 643474. EB was partially supported by  NSERC Discovery Grants 247764 and 643473. JL was partially supported by NSERC Discovery Grant 247764, FRQNT Grant 250479, and the Canada Research Chairs program. We thank Gavin Barill and Maeve Wildes for useful discussions pertaining to the convergence of finite difference schemes for Hamilton-Jacobi equations. 

\appendix
\section{An introduction to viscosity solutions}\label{app}
In this section, we provide a self-contained description of Crandall-Lions (continuous) and Barron-Jensen (lsc) viscosity solutions. The results of this section are classical and can be found in various references such as \cite{users, lionsbook, ABI, Bbook}. 

We will work throughout this section with the model equation 
\begin{equation}\label{e.ahj}
u_{t}+H(u_{x})=0, 
\end{equation}
where $H: \RR\rightarrow \RR$. We also define the Cauchy problem, given by
\begin{equation}\label{e.achj}
\begin{cases}
u_{t}+H(u_{x})=0&\text{in $\mathbb{R}\times (0, \infty)$},\\
u(x,0)=u_{0}(x)&\text{in $\mathbb{R}$.}
\end{cases}
\end{equation}

We begin with the theory of continuous viscosity solutions. 
\begin{define}
Let $u: \RR\times(0, \infty)\rightarrow \RR$. 
We say that $u$ is a viscosity {\em subsolution} of \eqref{e.ahj} at $(x_{0}, t_{0})$ if $u$ is upper semicontinuous at $(x_{0}, t_{0})$, and for any function 
$\vp\in C^{1}(\RR\times (0, \infty))$ such that $u-\vp$ has a local maximum at $(x_{0},t_{0})$, we have 
\begin{equation*}
\vp_{t}(x_{0},t_{0})+H(\vp_{x}(x_{0}, t_{0}))\leq 0.
\end{equation*}
We say that $u$ is a viscosity {\em supersolution} of \eqref{e.ahj} at $(x_{0}, t_{0})$ if $u$ is lower semicontinuous at $(x_{0}, t_{0})$, and for any function 
$\vp\in C^{1}(\RR\times (0, \infty))$ such that $u-\vp$ has a local minimum at $(x_{0},t_{0})$, we have
\begin{equation*}
\vp_{t}(x_{0},t_{0})+H(\vp_{x}(x_{0}, t_{0}))\ge 0.
\end{equation*}
Finally, we say that $u$ is a viscosity solution of \eqref{e.achj} if and only if $u$ is both a viscosity subsolution and supersolution of \eqref{e.ahj} for all $(x_{0}, t_{0})\in \RR\times (0, \infty)$ and, additionally, for all $x \in \R$, $u(y,t) \to u_0(x)$ as $(y,t) \to (x,0)$. As $u$ is then both upper and lower semicontinuous on $\RR\times (0, \infty)$, $u$ is necessarily continuous. 
\end{define}

One can also interpret the definition of viscosity solutions from a geometric perspective. The condition that $u-\vp$ has a local max/min at $(x_{0}, t_{0})$ can always be replaced by the function $\vp$ touching $u$ at the point $(x_{0}, t_{0})$ from above/below. Indeed, when $u-\vp$ has a local max at $(x_{0}, t_{0})$, we may adjust $\vp$ (adding appropriate constants and strictly convex/concave functions) to obtain $\tilde{\vp}$ such that 
\begin{equation*}
u<\tilde{\vp} \quad\text{in $\RR\times (0, \infty)$, except at $(x_0, t_{0})$ where $u(x_{0}, t_{0})=\tilde{\vp}(x_{0}, t_{0}).$} 
\end{equation*}

If $u$ is differentiable at $(x_{0}, t_{0})$ and satisfies \eqref{e.ahj} at $(x_{0}, t_{0})$, then $u$ automatically satisfies \eqref{e.ahj} in the viscosity sense at $(x_{0},t_{0})$. The notion of viscosity solution entails that if $u$ is not differentiable at $(x_{0}, t_{0})$, one uses a smooth test function which ``touches'' $u$ at the point $(x_{0}, t_{0})$ on either side to evaluate the PDE at $(x_{0}, t_{0})$. Compared to other notions of weak solutions of PDEs (for example, distributional solutions which are based on integration by parts), viscosity solutions are particularly amenable to \emph{nonlinear} PDEs. We now recall the basic existence and uniqueness result for continuous viscosity solutions which we will use throughout the paper: 

\begin{thm}\cite[Theorem VI.2]{CL}\label{t.cvisc}
Consider \eqref{e.achj} with $H$ continuous and $u_{0}$ bounded and uniformly continuous. There exists a unique continuous viscosity solution $u$ of \eqref{e.achj}. Moreover, 
\[
|u(x,t)-u(y,t)|\leq \sup_{\zeta\in \RR}|u_{0}(\zeta)+u_{0}(\zeta+y-x)|\quad\text{for $x,y\in \RR, t\geq 0$}. 
\]
\end{thm}
It is well known (see for example \cite[10.3, Theorem 3]{evans} that when $H(p)$ is convex and $\lim_{|p|\rightarrow \infty}\frac{H(p)}{|p|}=+\infty$, the unique continuous viscosity solution is given by the Hopf-Lax Formula 
\begin{equation*}
u(x,t)=\inf_{y\in \R} \left\{u_{0}(y)+tH^{*}\left(\frac{x-y}{t}\right)\right\}. 
\end{equation*}

The crown jewel of continuous viscosity solutions theory is the celebrated comparison principle, which is an extremely useful tool for analysis: 

\begin{thm}\cite[Theorem 8.2]{users}\label{t.comparison}
Consider \eqref{e.achj} with $H$ continuous. If $u$ is a subsolution of \eqref{e.ahj} and $v$ is a supersolution of \eqref{e.ahj}, 
and $u(x,0)=u_{0}(x)\leq v_{0}(x)=v(x,0)$ with $u_{0}, v_{0}$ bounded and uniformly continuous, then $u(x,t)\leq v(x,t)$ for all $t>0$. 
\end{thm}

Using the Comparison Principle (Theorem \ref{t.comparison}), we can show that $u$ solving \eqref{e.achj} satisfies additional regularity estimates: 
\begin{prop}\label{p.regularity}
Let $u$ denote the unique continuous viscosity solution of \eqref{e.achj} with $u_{0}$ bounded and Lipschitz continuous with Lipschitz constant $K>0$. Then there exists $C>0$ such that for all $(x,t)\in \RR\times (0, \infty)$, 
\begin{equation*}
\begin{cases}
|u_{t}(x,t)|\leq C,\\
|u_{x}(x,t)|\leq K.
\end{cases}
\end{equation*}
\end{prop}

\begin{proof}
The fact that  $|u_{x}|\leq K$ in all of $\RR\times (0, \infty)$ is automatic by Theorem \ref{t.cvisc}.  We now show that $u_{t}$ is uniformly bounded. In order to do so, we note that for $C:=\sup_{|p|\leq K}H(p)$, 
\begin{align*}
v(x,t):=u_{0}(x)+Ct \quad \mbox{ and }\quad
w(x,t):=u_{0}(x)-Ct
\end{align*}
are both super and subsolutions of \eqref{e.achj} respectively. Therefore, the Comparison Principle (Theorem \ref{t.comparison}) yields
\begin{equation*}
u_{0}(x)-Ct\leq u(x,t)\leq u_{0}(x)+Ct, 
\end{equation*}
which implies that 
\begin{equation}\label{e.intlip}
\sup_{t>0}\left|\frac{u(x,t)-u_{0}(x)}{t}\right|\leq C,
\end{equation}
for all $x\in \RR$. Now, for any $s > 0$, considering the function $u^{s}(x,t):=u(x,t+s)$, we have that 
\begin{equation*}
u^{s}(x,0)-\norm{u(x,0)-u^{s}(x,0)}_{L^{\infty}}\leq u(x,0)\leq u^{s}(x,0)+\norm{u(x,0)-u^{s}(x,0)}_{L^{\infty}}. 
\end{equation*}
Another application of the Comparison Principle (Theorem \ref{t.comparison}) implies that 
\begin{equation*}
u^{s}(x,t)-\norm{u(x,0)-u^{s}(x,0)}_{L^{\infty}}\leq u(x,t)\leq u^{s}(x,t)+\norm{u(x,0)-u^{s}(x,0)}_{L^{\infty}}, 
\end{equation*}
so that by \eqref{e.intlip}, 
\begin{equation*}
\left|u(x,t+s)-u(x,t)\right|\leq \norm{u(x,0)-u^{s}(x,0)}_{L^{\infty}}\leq Cs. 
\end{equation*}
This implies that $|u_{t}|\leq C$ for all $(x,t)\in \RR\times (0, \infty)$. 
\end{proof}

We now introduce the notion of Barron-Jensen or lower semicontinuous viscosity solutions, which is only defined when $H$ is a convex function. 

\begin{define} 
A lower semicontinuous function $u: \RR\times (0, \infty) \rightarrow\RR$ is an lsc viscosity solution of \eqref{e.ahj} at $(x_{0}, t_{0})$ if for every $\vp\in C^{1}(\RR\times (0, \infty))$ such that $u-\vp$ has a local minimum at $(x_{0}, t_{0})$, we have that 
\begin{equation*}
\vp_{t}(x_{0},t_{0})+H(\vp_{x}(x_{0}, t_{0}))=0.
\end{equation*}
We say that $u$ is a lsc solution of \eqref{e.achj} if $u$ is a lsc viscosity solution for all $(x_{0}, t_{0})\in \RR\times (0, \infty)$ and \begin{equation*}
\inf\left\{\liminf_{n\rightarrow \infty} u(x_{n}, t_{n}) \mid t_{n}\rightarrow 0, x_{n}\rightarrow x\right\}=u_{0}(x). 
\end{equation*}
\end{define}

In the case when $u$ is continuous, we have an equivalence between the two definitions:
\begin{thm}\cite[Theorem 16]{Bbook}\label{t.equiv}
Assume $H$ is convex. A continuous function is a viscosity solution of \eqref{e.ahj} if and only if it is a lsc viscosity solution \eqref{e.ahj}. 
\end{thm}

Finally, we recall that in the case when $H$ is convex, a natural candidate for a solution (from the point of view of optimal control) is the solution given by the Hopf-Lax formula. In the cases when $u_{0}$ is lower semicontinuous and bounded below, the Hopf-Lax formula gives rise to the unique lsc viscosity solution. 

\begin{thm}\cite[Theorem 5.2]{ABI}\label{t.bjhl}
Let $u_{0}: \RR\rightarrow \RR$ be lsc with 
\begin{equation*}
u_{0}(x)\geq -C(|x|+1).
\end{equation*}
 Let $H: \RR\rightarrow \RR$ be convex and Lipschitz. Then 
 \begin{equation*}
 u(x,t)=\inf_{y\in \RR} \left\{u_{0}(y)+tH^{*}\left(\frac{x-y}{t}\right)\right\}
 \end{equation*}
 is the unique lsc viscosity solution of \eqref{e.achj} bounded from below by a function of linear growth. 
\end{thm}

\bibliographystyle{acm}
\bibliography{general_HRW}

\end{document}